\documentclass{amsart}

\usepackage{enumitem}
\usepackage{amssymb}

\newtheorem{thm}{Theorem}[section]
\newtheorem{lem}[thm]{Lemma}
\newtheorem{cor}[thm]{Corollary}
\newtheorem{defi}[thm]{Definition}

\newtheorem{nota}[thm]{Notation}
\newtheorem{quest}[thm]{Question}
\newtheorem{ex}[thm]{Example}
\newtheorem{fact}[thm]{Fact}

\numberwithin{equation}{section}

\begin{document}
	
\title{$C^{\ast}$-algebraic approach to the principal symbol. III}

\author{Y. Kordyukov}
\address{Institute of Mathematics, Ufa Federal Research Centre, Russian Academy of Sciences, 112~Chernyshevsky str., 450008 Ufa, Russia.} \email{yurikor@matem.anrb.ru}

\author{F. Sukochev}
\address{School of Mathematics and Statistics, University of New South Wales, Kensington, 2052, NSW, Australia.}
\email{f.sukochev@unsw.edu.au}

\author{D. Zanin}
\address{School of Mathematics and Statistics, University of New South Wales, Kensington, 2052, NSW, Australia.}
\email{d.zanin@unsw.edu.au}



\date{}

\begin{abstract} We treat the notion of principal symbol mapping on a compact smooth manifold as a $\ast$-homomorphism of $C^{\ast}$-algebras. Principal symbol mapping is built from the ground, without referring to the pseudodifferential calculus on the manifold. Our concrete approach allows us to extend Connes Trace Theorem for compact Riemannian manifolds.
\end{abstract}

\maketitle

\section{Introduction}

This paper is motivated by the theory of pseudodifferential operators. A central notion of that theory is that of a principal symbol, which is roughly a homomorphism from the algebra of pseudo-differential operators into an algebra of functions, \cite[Lemma 5.1]{Kohn-Nirenberg}, \cite[Theorem 5.5]{Joshi}, \cite[pp. 54-55]{Treves1}. Usually, it is defined in a manner inhospitable for operator theorists. However, in \cite{DAO1}, a new approach to a principal symbol mapping on a certain $C^{\ast}$-subalgebra $\Pi$ in $B(L_2(\mathbb{R}^d))$ is proposed; this mapping turns out to be a $\ast$-homomorphism from $\Pi$ into a commutative $C^{\ast}$-algebra. The $C^{\ast}$-algebra $\Pi$ contains all classical compactly based pseudodifferential operators. This provides a very simple and algebraic approach to the theory.

Whereas our approach is more elementary than the classical approach, the $C^{\ast}$-algebra $\Pi$ introduced in \cite{DAO1} (see also \cite{DAO2}) is much wider than the class of a classical compactly based pseudo-differential operators of order $0$ on $\mathbb{R}^d.$  The aim of this paper is to extend this $C^{\ast}$-algebraic approach to the setting of smooth compact manifolds.

The $C^{\ast}$-algebra $\Pi$ in the Definition \ref{pis def} below is the closure (in the uniform norm) of the $\ast$-algebra of all compactly supported \textit{classical} pseudodifferential operators of order $0.$ However, we use an elementary definition of $\Pi$ which does not involve pseudodifferential operators. The idea to consider this closure may be discerned yet in \cite{Atiyah-Singer} (see Proposition 5.2 on p.512). For the recent development of this idea we refer to \cite{DAO1,DAO2}.

Let $D_k=\frac{\partial}{i\partial t_k}$ be the $k-$th partial derivative operator on $\mathbb{R}^d$ (these are unbounded self-adjoint operators on $L_2(\mathbb{R}^d)$). In what follows, $\nabla=(D_1,\cdots,D_d)$ and $\Delta=\sum_{k=1}^d\frac{\partial^2}{\partial^2t_k}=-\sum_{k=1}^dD_k^2.$ Let the $d-$dimensional vector $\frac{\nabla}{(-\Delta)^{\frac12}}$ be defined by the functional calculus. Let $M_f$ be the multiplication operator by the function $f.$

\begin{defi}\label{pis def} Let $\pi_1:L_{\infty}(\mathbb{R}^d)\to B(L_2(\mathbb{R}^d)),$ $\pi_2:L_{\infty}(\mathbb{S}^{d-1})\to B(L_2(\mathbb{R}^d))$ be defined by setting
\[\pi_1(f)=M_f,\quad \pi_2(g)=g(\frac{\nabla}{\sqrt{-\Delta}}),\quad f\in L_{\infty}(\mathbb{R}^d),\quad g\in L_{\infty}(\mathbb{S}^{d-1}).\]
Let $\mathcal{A}_1=\mathbb{C}+C_0(\mathbb{R}^d)$ and $\mathcal{A}_2=C(\mathbb{S}^{d-1}).$ Let $\Pi$ be the $C^{\ast}$-subalgebra in $B(L_2(\mathbb{R}^d))$ generated by the algebras $\pi_1(\mathcal{A}_1)$ and $\pi_2(\mathcal{A}_2).$
\end{defi}

According to \cite{DAO1}, there exists an $\ast$-homomorphism
\begin{equation}\label{euclidean symbol definition}
\mathrm{sym}:\Pi\to\mathcal{A}_1\otimes_{\mathrm{ min}}\mathcal{A}_2\simeq C(\mathbb{S}^{d-1},\mathbb{C}+C_0(\mathbb{R}^d))
\end{equation}
such that
\[\mathrm{sym}(\pi_1(f))=f\otimes 1,\quad \mathrm{sym}(\pi_2(g))=1\otimes g.\]
Here, $\mathcal{A}_1\otimes_{\mathrm{min}}\mathcal{A}_2$ is the minimal tensor product of the $C^{\ast}$-algebras $\mathcal{A}_1$ and $\mathcal{A}_2$ (see Propositions 1.22.2 and 1.22.3 in \cite{Sakai-book}). Elements of $\mathcal{A}_1\otimes_{\mathrm{min}}\mathcal{A}_2$ are identified with continuous functions on $\mathbb{R}^d\times\mathbb{S}^{d-1}.$ This $\ast$-homomorphism is called a principal symbol mapping. It properly extends the notion of the principal symbol of the classical pseudodifferential operator.

%

It is natural to ask whether $C^{\ast}$-algebraic approach works in the general setting of smooth compact manifolds. It makes sense to de-manifoldize the question and reformulate it in a purely Euclidean fashion. We begin with the natural question on the properties of the $C^{\ast}$-algebra $\Pi.$

\begin{quest}\label{higson question} The natural unitary action of the group of diffeomorphisms on $\mathbb{R}^d$ is defined as follows. Let $\Phi:\mathbb{R}^d\to\mathbb{R}^d$ be a diffeomorphism. Let $U_{\Phi}\in B(L_2(\mathbb{R}^d))$ be a unitary operator given by setting
\[U_{\Phi}\xi=|\mathrm{det}(J_{\Phi})|^{\frac12}\cdot (\xi\circ\Phi),\quad \xi\in L_2(\mathbb{R}^d).\]	
Here, $J_{\Phi}$ is the Jacobian matrix of $\Phi.$	

Is the $C^{\ast}$-algebra $\Pi$ invariant under the action $T\to U_{\Phi}^{-1}TU_{\Phi}?$ Does the $\ast$-homomorphism $\mathrm{sym}$ behave equivariantly under this action?
\end{quest}

Theorem \ref{algebra invariance special} provides a positive answer to Question \ref{higson question} (under the additional requirement that $\Phi$ is affine outside of some ball). This additional assumption yields, in particular, that $\Phi$ extends to a diffeomorphism of the projective space $P^d(\mathbb{R}).$ We emphasise that the Question \ref{higson question} in full generality remains open. Furthermore, Theorem \ref{algebra invariance theorem} proves an invariance of $\Pi$ and equivariance of $\mathrm{sym}$ under local diffeomorphisms. 

The resolution of Question \ref{higson question} has opened an avenue for the definition of the $C^{\ast}$-algebra $\Pi_X$ associated with an arbitrary compact smooth manifold $X.$ This $C^{\ast}$-algebra has a remarkable property: it admits a $\ast$-homomorphism $\mathrm{sym}_X:\Pi_X\to C(S^{\ast}X),$ where $S^{\ast}X$ is the cosphere bundle of $X$ (see the Subsection \ref{cosphere subsection}). If $X=\mathbb{R}^d,$ then $\mathrm{sym}_X$ coincides with the mapping $\mathrm{sym}$ above. Every \textit{classical} order $0$ pseudodifferential operator $T$ on $X$ belongs to $\Pi_X$ and its principal symbol in the sense of pseudodifferential operators equals $\mathrm{sym}_X(T).$ On the other hand, not every element of $\Pi_X$ is pseudodifferential (e.g. because principal symbol of a pseudodifferential operator is necessarily smooth, while that of element of $\Pi_X$ is only continuous). An approach to pseudodifferential calculi based on $C^*$-algebras theory was first suggested by H.O. Cordes \cite{Cordes87} (see \cite{melo05} for the case of a closed manifold).

Below, we briefly describe the construction of $\Pi_X$ via the patching process (see more precise description in Subsection \ref{construction subsection}).

Let $X$ be a compact smooth manifold with an atlas $(\mathcal{U}_i,h_i)_{i\in\mathbb{I}}.$ We will fix a sufficiently good measure $\nu$ on $X$, given by a continuous positive density (see Definition \ref{condition on nu}). If $T\in B(L_2(X,\nu))$ is compactly supported in some chart $(\mathcal{U}_i,h_i)$ (i.e., there exists $\phi\in C^{\infty}_c(\mathcal{U}_i)$ such that $T=TM_{\phi}=M_{\phi}T$), then, by composing with $h_i,$ we can transfer $T$ to an operator on $L_2(\mathbb{R}^d).$

\begin{defi}\label{rough principal symbol for compact manifolds} Let $X$ be a compact smooth manifold equipped with a continuous positive density $\nu$ and let $T\in B(L_2(X,\nu)).$ We say that $T\in\Pi_X$ if
\begin{enumerate}
\item for every $i\in\mathbb{I}$ and for every $\phi\in C_c(\mathcal{U}_i),$ the operator $M_{\phi}TM_{\phi}$ transferred to an operator on $L_2(\mathbb{R}^d)$ belongs to $\Pi;$
\item for every $\psi\in C(X),$ the operator $[T,M_{\psi}]$ is compact.
\end{enumerate}
\end{defi}

\begin{thm}\label{symbol manifold intro thm} If $X$ is a smooth compact manifold and if $\nu$ is a continuous positive density on $X,$ then $\Pi_X$ is a $C^{\ast}$-algebra and there exists (see Definition \ref{principal symbol definition thm}) a surjective $\ast$-homomorphism \[\mathrm{sym}_X:\Pi_X\to C(S^{\ast}X)\]
such that
\[\mathrm{ker}(\mathrm{sym}_X)=\mathcal{K}(L_2(X,\nu)).\]
\end{thm}  

In other words, we have a short exact sequence
\[0\to\mathcal{K}(L_2(X,\nu))\stackrel{\mathrm{id}}{\to}\Pi_X\stackrel{\mathrm{sym}_X}{\to} C(S^{\ast}X)\to0.\]

This short exact sequence first appeared in \cite{Atiyah-Singer} (see Proposition 5.2 on p.512) and plays an important role in index theory (see, for instance, \cite[Section 24.1.8]{Blackadar-book} or \cite[Section 2]{Baum-Douglas}).  It is essentially equivalent to the fact that for any operator $T\in \Pi_X$ with principal symbol $a\in C(S^{\ast}X)$,
\[\inf\{\|T+K\|_{\infty}:\ K \in \mathcal{K}(L_2(X,\nu))\}=\|a\|_{C(S^{\ast}X)}.\]
For singular integral operators this result was proved by Gohberg \cite{Gohberg} and Seeley \cite{Seeley}. Proofs in the language of pseudodifferential operators have been given in \cite{Hormander67,Kohn-Nirenberg}. It should be noted that the definition given in \cite{Atiyah-Singer} is somewhat imprecise (see \cite{melo05}, in particular, a discussion on p. 329). 


As a corollary of Theorem \ref{symbol manifold intro thm}, we provide a version of Connes Trace Theorem (see Theorem \ref{ctt manifold} below). As stated, it extends Theorem 1 in \cite{Connes-action}. Connes Trace Theorem is ubiquitous in Non-commutative Geometry. It serves as a ground for defining a general notion of the non-commutative integral and non-commutative Yang-Mills action (that is, Theorem 14 in \cite{Connes-action} is taken as a definition in the non-commutative setting). 

We now compare our Theorem \ref{ctt manifold} with various versions of Connes Trace Theorem available in the literature. Original proof of Connes was, according to \cite{GVF-book} "somewhat telegraphic". For example, it was not mentioned in \cite{Connes-action} that the manifold is Riemannian and that pseudodifferential operator featuring in Theorem 1 in \cite{Connes-action} is classical. Two proofs are given in \cite{GVF-book} (Theorem 7.18 on p.293) and both of them rely on the assumption of ellipticity of the underlying pseudo-differential operator (this assumption is redundant as demonstrated in our approach). Despite their critique of Connes exposition, the authors of \cite{GVF-book} also do not mention the classicality of their pseudodifferential operator. Another two proofs are given in \cite{AM-book}. As authors of \cite{AM-book} admit, their proofs are quite sketchy, however, they provide a correct statement. The advantage of our approach is threefold: (a) we consider a strictly larger class of operators (b) we consider a strictly larger class of traces (c) we work in a convenient category of $C^{\ast}$-algebras (i.e., non-commutative topological spaces) and not in a category of classical pseudodifferential operators which does not have a natural counterpart in Non-commutative Geometry.

\begin{thm}\label{ctt manifold} Let $\varphi$ be a normalised continuous trace on $\mathcal{L}_{1,\infty}.$ Let $(X,G)$ be a compact Riemannian manifold and let $\nu$ be the Riemannian volume. If $T\in\Pi_X,$ then
\[\varphi(T(1-\Delta_G)^{-\frac{d}{2}})=c_d\int_{T^{\ast}X}\mathrm{sym}_X(T)e^{-q_X}d\lambda,\]
where $\lambda$ is the Liouville measure on $T^{\ast}X$ and $e^{-q_X}$ is the canonical weight of the Riemannian manifold (as defined in Subsection \ref{canonical weight subsection}). 
\end{thm}

When $T$ is a classical pseudodifferential operator, the right hand side coincides with Wodzicki residue of $T(1-\Delta_G)^{-\frac{d}{2}}.$ We refer the reader to the extensive discussion of this matter in \cite{LSZ-obzor}.

One should note a sharp contrast between the setting of Theorem \ref{ctt manifold} and that of Theorem \ref{symbol manifold intro thm}. Indeed, in the latter theorem, the (smooth compact) manifold is rather arbitrary, while in the former it is Riemannian. The Riemannian structure of $X$ in Theorem \ref{ctt manifold} is needed in two places: (a) there is no natural measure on the cosphere bundle of an arbitrary smooth manifold (but such a measure arises naturally if the manifold is Riemannian) (b) Riemannian structure provides us with a natural second order differential operator (i.e, Laplace-Beltrami operator). In the setting of a general smooth manifoild, the second issue can be circumvented by replacing $\Delta_G$ with an arbitrary elliptic second order differential operator (whose resolvent falls into the ideal $\mathcal{L}_{d,\infty}$). However, the lack of a natural measure on $S^{\ast}X$ prevents us from stating Theorem \ref{ctt manifold} in that generality.

We now briefly describe the structure of the paper. Section \ref{prelim section} collects known facts used further in the text. Theorems \ref{algebra invariance special} and \ref{algebra invariance theorem} in Section \ref{diff invar section} assert equivariant behavior of the principal symbol mapping in Euclidean setting under the action of diffeomorphisms. Theorem \ref{algebra invariance special} is proved in Section \ref{ais section}. Theorem \ref{algebra invariance theorem} is proved in Section \ref{ait section}. Our main result, Theorem \ref{symbol manifold intro thm} is proved in Section \ref{symbol construction section} with the help of Globalisation Theorem from Subsection \ref{globalisation subsection} (proved in Appendix \ref{appendix section}). Finally, Connes Trace Theorem on compact Riemannian manifolds (that is, Theorem \ref{ctt manifold}) is proved in Section \ref{ctt section}.

\section{Preliminaries and notations}\label{prelim section}

As usual, $B(H)$ denotes the $\ast$-algebra of all bounded operators on the Hilbert space $H$ and  $\mathcal{K}(H)$ denotes the ideal of all compact operators in $B(H).$ As usual, Euclidean length of a vector $t\in\mathbb{R}^d$ is denoted by $|t|.$

We frequently use the equality
\begin{equation}\label{dk mf commutator eq}
[D_k,M_f]=M_{D_kf},\quad f\in C^{\infty}(\mathbb{R}^d).
\end{equation}

\subsection{Principal ideals in $B(H)$}

It is well known that every ideal in $B(H)$ consists of compact operators. 

Undoubtedly, the most important ideals are the principal ones. Among them, a special role is played by the ideal $\mathcal{L}_{p,\infty},$ a principal ideal generated by the operator $\mathrm{diag}(((k+1)^{-\frac1p})_{k\geq0}).$ We frequently use the following property (related to the H\"older inequality) of this scale of ideals 
\[\mathcal{L}_{p,\infty}\cdot\mathcal{L}_{q,\infty}=\mathcal{L}_{r,\infty},\quad \frac1r=\frac1p+\frac1q.\]
We mention in passing that $\mathcal{L}_{p,\infty}$ is quasi-Banach for every $p>0$ (however, we do not need the quasi-norms in this text).

\subsection{Traces on $\mathcal{L}_{1,\infty}$}

\begin{defi}\label{trace def} If $\mathcal{I}$ is an 
ideal in $B(H),$ then a unitarily invariant linear functional $\varphi:\mathcal{I}\to\mathbb{C}$ is said to be a trace.
\end{defi}
Since $U^{-1}TU-T=[U^{-1},TU]$ for all $T\in\mathcal{I}$ and 
for all unitaries $U\in B(H),$ and since the unitaries 
span $B(H),$ it follows that traces are precisely 
the linear functionals on $\mathcal{I}$ satisfying the condition
\[\varphi(TS)=\varphi(ST),\quad T\in\mathcal{I}, S\in B(H).\]
The latter may be reinterpreted as the vanishing of the 
linear functional $\varphi$ on the commutator 
subspace which is denoted $[\mathcal{I},B(H)]$ and defined to be the linear 
span of all commutators $[T,S]:\ T\in\mathcal{I},$ $S\in B(H).$ 
Note that  
$\varphi(T_1)=\varphi(T_2)$ whenever 
$0\leq T_1,T_2\in\mathcal{I}$ are such that the singular value 
sequences $\mu(T_1)$ and $\mu(T_2)$ coincide.
For $p>1,$ the ideal $\mathcal{L}_{p,\infty}$ does not admit a 
non-zero trace while for $p=1,$ 
there exists a plethora of traces on $\mathcal{L}_{1,\infty}$ 
(see e.g. \cite{LSZ-book}). An example of  a trace on 
$\mathcal{L}_{1,\infty}$ is the Dixmier trace introduced in \cite{Dixmier}
that we now explain.

\begin{ex} Let $\omega$ be an extended limit.
Then the functional $\mathrm{Tr}_{\omega}:\mathcal{L}_{1,\infty}^+\to\mathbb{R}_+$ defined by setting
\[\mathrm{Tr}_{\omega}(A)=\omega\Big(\Big\{\frac1{\log(2+n)}\sum_{k=0}^n\mu(k,A)\Big\}_{n\geq0}\Big),
\quad 0\leq A\in\mathcal{L}_{1,\infty},\]
is additive and, therefore, extends to a trace on $\mathcal{L}_{1,\infty}.$ We call such traces  \textit{Dixmier traces}. These traces clearly depend on the choice of the functional $\omega$ on $l_\infty.$
\end{ex}

An extensive discussion of traces, and more recent developments in the theory,
may be found in \cite{LSZ-book} including a discussion of the following facts.
\begin{enumerate}
\item All Dixmier traces on $\mathcal{L}_{1,\infty}$ are positive.
\item All positive traces on $\mathcal{L}_{1,\infty}$ are continuous in the quasi-norm topology.
\item There exist positive traces on $\mathcal{L}_{1,\infty}$ which are not Dixmier traces.
\item There exist traces on $\mathcal{L}_{1,\infty}$ which fail to be continuous.
\end{enumerate}

We are mostly interested in \textit{normalised traces} $\varphi:\mathcal{L}_{1,\infty}\to\mathbb{C},$ that is, satisfying $\varphi(T)=1$ whenever $0\leq T$ is such that $\mu(k,T)=\frac1{k+1}$ for all $k\geq0.$

Traces on $\mathcal{L}_{1,\infty}$ play a fundamental role in Non-commutative Geometry. For example, they allow to write Connes Character Formula (we refer the reader to Section 5.3 in \cite{LSZ-obzor} and references therein).

\subsection{Sobolev spaces}

Sobolev space $W^{m,2}(\mathbb{R}^d),$ $m\in \mathbb{Z}_+,$ consists of all distributions $f\in L_2(\mathbb{R}^d)$ such that every distributional derivative $D^{\alpha}f,$ $\alpha\in \mathbb{Z}_+^d,$ of order 
$|\alpha|_1=\sum_{k=1}^d\alpha_k\leq m$ also belongs to $L_2(\mathbb{R}^d).$ 

The importance of Sobolev spaces in the theory of differential operators can be seen e.g. from the fact that $W^{1,2}(\mathbb{R}^d)$ is the domain of the self-adjoint tuple $\nabla.$ Also, $W^{2,2}(\mathbb{R}^d)$ is the domain of the self-adjoint positive operator $-\Delta.$

We refer the reader to the books \cite{Adams-book}, \cite{Taylor-book} for further information on Sobolev spaces.

Further we need the following standard result (see e.g. p.322 in \cite{Taylor-book}).

\begin{thm}\label{322 taylor} Sobolev space $W^{m,2}(\mathbb{R}^d),$ $m\in\mathbb{Z}_+,$ is invariant under diffeomorphisms which are affine outside of some ball.
\end{thm} 

\subsection{Pseudodifferential operators}\label{psdo subsection}

If $p\in C^{\infty}(\mathbb{R}^d\times\mathbb{R}^d)$ (that is, bounded smooth function whose derivatives are also bounded functions), then the Calderon-Vaillancourt theorem (see e.g. unnumbered proposition on p.282 in \cite{Stein-book}) the operator $\mathrm{Op}(p)$ defined by the formula (here $\mathcal{F}$ is Fourier transform on $\mathbb{R}^d$)
\begin{equation}\label{psdo def}
(\mathrm{Op}(p)\xi)(t)=(2\pi)^{-\frac{d}{2}}\int_{\mathbb{R}^d}e^{i\langle t,s\rangle}p(t,s)(\mathcal{F}\xi)(s)ds,\quad \xi\in L_2(\mathbb{R}^d),
\end{equation}
is bounded in $L_2(\mathbb{R}^d).$

If $m\in\mathbb{Z},$ $m\leq 0$ is such that
\begin{equation}\label{psim p condition}
\sup_{t,s\in\mathbb{R}^d}(1+|s|^2)^{\frac{|\beta|_1-m}{2}}|D_t^{\alpha}D_s^{\beta}p(t,s)|<\infty,\quad \alpha,\beta\in\mathbb{Z}_+^d,
\end{equation}
then we say that $\mathrm{Op}(p)\in\Psi^m(\mathbb{R}^d).$ For $m>0,$ the class $\Psi^m(\mathbb{R}^d)$ is defined by the same formula. The difference is that, for $m>0,$ operators in $\Psi^m(\mathbb{R}^d)$ are no longer bounded as operators from $L_2(\mathbb{R}^d)\to L_2(\mathbb{R}^d);$ instead, they are bounded operators from $W^{m,2}(\mathbb{R}^d)$ to $L_2(\mathbb{R}^d).$

The key property is that
\begin{equation}\label{psi product eq}
\Psi^{m_1}(\mathbb{R}^d)\cdot\Psi^{m_2}(\mathbb{R}^d)\subset\Psi^{m_1+m_2}(\mathbb{R}^d),\quad m_1,m_2\in\mathbb{Z}.
\end{equation}
Moreover, by Theorem 2.5.1 in \cite{RT-book}, we have
\begin{equation}\label{op product eq}
\mathrm{Op}(p_1)\cdot\mathrm{Op}(p_2)\in \mathrm{Op}(p_1p_2)+\Psi^{m_1+m_2-1}(\mathbb{R}^d),
\end{equation}
whenever $\mathrm{Op}(p_1)\in\Psi^{m_1}(\mathbb{R}^d)$ and $\mathrm{Op}(p_2)\in\Psi^{m_2}(\mathbb{R}^d).$ The next lemma follows immediately from \eqref{psi product eq} and \eqref{op product eq}.

\begin{lem}\label{psi product lemma} If $m_1,m_2\in\mathbb{Z},$
\[T_l\in \mathrm{Op}(p_l)+\Psi^{m_l-1}(\mathbb{R}^d),\quad \mathrm{Op}(p_l)\in\Psi^{m_l}(\mathbb{R}^d),\quad l=1,2,\]
then
\[T_1T_2-\mathrm{Op}(p_1p_2)\in\Psi^{m_1+m_2-1}(\mathbb{R}^d).\]
\end{lem}

Let $T\in\Psi^m(\mathbb{R}^d),$ $m<0,$ and let $\psi\in C^{\infty}_c(\mathbb{R}^d)$ be such that $T=M_{\psi}T.$ Recall that the operator $M_{\psi}(1-\Delta)^{\frac{m}{2}}$ is compact (see e.g. Theorem 4.1 in \cite{Simon-book}). Thus,
\begin{equation}\label{compactly supported are compact}
T=M_{\psi}(1-\Delta)^{\frac{m}{2}}\cdot (1-\Delta)^{-\frac{m}{2}}T\in \mathcal{K}(L_2(\mathbb{R}^d))\cdot B(L_2(\mathbb{R}^d))=\mathcal{K}(L_2(\mathbb{R}^d)).
\end{equation}

Differential operators of order $m\geq0$ with smooth bounded coefficients (all derivatives of the coefficients are also assumed bounded) belong to $\Psi^m(\mathbb{R}^d).$ Indeed, it follows directly from \eqref{psdo def} that
\begin{equation}\label{do vs psdo}
\sum_{|\alpha|_1\leq m}M_{f_{\alpha}}D^{\alpha}=\mathrm{Op}(p),\quad p(t,s)=\sum_{|\alpha|_1\leq m}f_{\alpha}(t)s^{\alpha},\quad t,s\in\mathbb{R}^d.
\end{equation}

The following standard result is available, e.g. in Theorem 1.6.20 in \cite{LMSZ-book}. 

\begin{thm}\label{complex power thm} Let $T\in\Psi^m(\mathbb{R}^d),$ $m\geq 0,$ extend to a self-adjoint positive operator $T:W^{m,2}(\mathbb{R}^d)\to L_2(\mathbb{R}^d).$ For every $z\in\mathbb{C},$ we have $(T+1)^z\in\Psi^{m\Re(z)}(\mathbb{R}^d).$
	
If, in addition, $T$ is a differential operator with positive principal symbol $p,$ then
\[(T+1)^z-\mathrm{Op}((p+1)^z)\in\Psi^{m\Re(z)-1}(\mathbb{R}^d).\]	
\end{thm}

\subsection{Pseudodifferential-like operators in $\Pi$}\label{psdo-like subsection}

If $q\in C^{\infty}_c(\mathbb{R}^d\times\mathbb{S}^{d-1}),$ then we set
\begin{equation}\label{psdo-like operator def}
(T_q\xi)(t)=(2\pi)^{-\frac{d}{2}}\int_{\mathbb{R}^d}e^{i\langle t,s\rangle}q(t,\frac{s}{|s|})(\mathcal{F}\xi)(s)ds,\quad\xi\in L_2(\mathbb{R}^d).
\end{equation}

\begin{lem}[Lemma 8.1 in \cite{DAO1}]\label{81 lemma} For every $q\in C^{\infty}_c(\mathbb{R}^d\times\mathbb{S}^{d-1}),$ we have  $T_q\in \Pi$ and $\mathrm{sym}(T_q)=q.$
\end{lem}

\begin{lem}[Lemma 8.2 in \cite{DAO1}]\label{82 lemma} Let $q\in C^{\infty}_c(\mathbb{R}^d\times\mathbb{S}^{d-1}).$ If  $\psi\in C^{\infty}_c(\mathbb{R}^d)$ equals $1$ near $0,$ then
\[\mathrm{Op}(p)-T_q\in\mathcal{K}(L_2(\mathbb{R}^d)),\]
where
\[p(t,s)=q(t,\frac{s}{|s|})\cdot (1-\psi(s)),\quad t,s\in\mathbb{R}^d.\]
\end{lem}

\subsection{Cotangent bundle}

\begin{nota}\label{manifold nota} Let $X$ be a smooth $d$-dimensional manifold with atlas $(\mathcal{U}_i,h_i)_{i\in\mathbb{I}},$ where $\mathbb{I}$ is an arbitrary set of indices.
\begin{enumerate}
\item We denote
\[\Omega_i=h_i(\mathcal{U}_i)\subset\mathbb{R}^d,\quad \Omega_{i,j}=h_i(\mathcal{U}_i\cap\mathcal{U}_j)\subset\mathbb{R}^d,\quad i,j\in\mathbb{I};\]
\item We denote by $\Phi_{i,j}:\Omega_{i,j}\to\Omega_{j,i}$ the diffeomorphism given by the formula
\[\Phi_{i,j}(t)=h_j(h_i^{-1}(t)),\quad t\in\Omega_{i,j}.\]
\end{enumerate} 
\end{nota}

In the next fact, we recall the manifold structure of $T^{\ast}X.$

\begin{fact}\label{cotangent bundle mapping properties thm} Let $X$ be a $d$-dimensional manifold with an atlas $\{(\mathcal{U}_i,h_i)\}_{i\in\mathbb{I}}.$ Let $T^{\ast}X$ be the cotangent bundle of $X$ and let $\pi:T^{\ast}X\to X$ be the canonical projection. There exists an atlas $\{\pi^{-1}(\mathcal{U}_i),H_i\}_{i\in\mathbb{I}}$ of $T^{\ast}X$ such that
\begin{enumerate}
\item for every $i\in\mathbb{I},$ $H_i:\pi^{-1}(\mathcal{U}_i)\to \Omega_i\times\mathbb{R}^d$ is a homeomorphism;
\item for every $i,j\in\mathbb{I}$ such that $\mathcal{U}_i\cap \mathcal{U}_j\neq\varnothing,$ we have
\[(H_j\circ H_i^{-1})(t,s)=(\Phi_{i,j}(t),(J^{\ast}_{\Phi_{i,j}}(t))^{-1}s),\quad t\in \Omega_{i,j},\quad s\in\mathbb{R}^d.\]
\end{enumerate}
\end{fact}

In the next fact, we identify functions on the $T^{\ast}X$ and their local representations. It is important that this identification preserves continuity.

\begin{fact}\label{germ def} Let $F_i:\Omega_i\times\mathbb{R}^d\to\mathbb{C}$ for every $i\in\mathbb{I}.$ If
\[F_i\circ H_i=F_j\circ H_j\mbox{ on }\pi^{-1}(\mathcal{U}_i\cap\mathcal{U}_j)\mbox{ for every }i,j\in\mathbb{I},\]
then there exists a unique function  $F:T^{\ast}X\to\mathbb{C}$ such that
\[F_i=F\circ H_i^{-1}\mbox{ on } \Omega_i\times \mathbb{R}^d,\quad i\in\mathbb{I}.\]
\end{fact}

\subsection{Cosphere bundle}\label{cosphere subsection}

\begin{defi}\label{dilation action} Define a dilation action $\lambda\to\sigma_{\lambda}$ of $(0,\infty)$ on each $\Omega_i\times\mathbb{R}^d$ by setting
\[\sigma_{\lambda}:(t,s)\to (t,\lambda s),\quad t\in 
\Omega_i, \quad s\in\mathbb{R}^d.\]
This action lifts down to an action on $T^{\ast}X$ (also denoted by $\sigma_{\lambda}$).

A function on $T^{\ast}X$ invariant with respect to this action is called dilation invariant.
\end{defi}

\begin{defi}\label{cosphere bundle def} Let $X$ be a compact manifold. $C^{\ast}$-algebra of all continuous dilation invariant functions on $T^{\ast}X\backslash 0_{T^{\ast}X}$ (here, $0_{T^{\ast}X}$ is the zero section of $T^{\ast}X$) is denoted by $C(S^{\ast}X)$ and is called the algebra of continuous functions on the cosphere bundle of $X.$ 
\end{defi}

\subsection{Canonical weight of Riemannian manifold}\label{canonical weight subsection}

If $X$ is a smooth $d$-dimensional manifold, then $T^{\ast}X$ has a canonical symplectic structure. The corresponding Liouville measure $\lambda$ on $T^{\ast}X$ satisfies the following property (see, for instance, \cite{Chavel}):
\begin{equation}\label{liouville measure defining eq}
\int_{\mathbb{R}^d\times\mathbb{R}^d}fdm=\int_{T^{\ast}X}(f\circ H_i)d\lambda,\quad f\in C_c(\Omega_i\times\mathbb{R}^d),\quad i\in\mathbb{I}.
\end{equation}
Here, $f\circ H_i$ denotes a function on $T^{\ast}X$ which equals $f\circ H_i$ on $\pi^{-1}(\mathcal{U}_i)$ and which vanishes outside $\pi^{-1}(\mathcal{U}_i).$

However, there is no canonical way to equip the cosphere bundle $S^{\ast}X$ of a smooth manifold $X$ with a measure. The following class of measures is of particular interest: if $w\in L_1(T^{\ast}X,\lambda),$ then the functional
\[f\to \int_{T^{\ast}X}fwd\lambda,\quad f\in C(S^{\ast}X),\]
generates a measure on $S^{\ast}X$ by the Riesz-Markov theorem. However, there is no canonical way to select an integrable function $w$ on $T^{\ast}X.$ This choice becomes possible if we assume in addition a Riemannian structure on $X.$

Let $G$ be a Riemannian metric on $X$. For any $i\in\mathbb{I}$, the components of the metric $G$ in the chart $(\mathcal{U}_i,h_i)$ give rise to a smooth mapping $G_i:\mathcal{U}_i\to \mathrm{GL}^+(d,\mathbb{R})$. (In what follows, $\mathrm{GL}^+(d,\mathbb{R})$ stands for the set of all positive elements in $\mathrm{GL}(d,\mathbb{R}).$)
For any $i,j\in\mathbb{I}$ such that $\mathcal{U}_i\cap \mathcal{U}_j\neq\varnothing,$ we have
\[G_j(t)=J_{\Phi_{j,i}}^{\ast}(h_j(t))\cdot G_i(t)\cdot J_{\Phi_{j,i}}(h_j(t)),\quad t\in\mathcal{U}_i\cap\mathcal{U}_j.\]
Here, $\Phi_{j,i}$ are given in Notation \ref{manifold nota}.

\begin{nota}\label{qi nota} For every $i\in\mathbb{I},$ let $\Omega_i$ be as in Notation \ref{manifold nota} and set $g_i=G_i\circ h_i^{-1}:\Omega_i\to \mathrm{GL}^+(d,\mathbb{R}).$ We also set
\[q_i(t,s)=\langle g_i(t)^{-1}s,s\rangle,\quad t\in\Omega_i,\quad s\in\mathbb{R}^d.\]
\end{nota}

It can be easily verified by a direct calculation that, for every $i,j\in\mathbb{I}$, we have
\[q_i\circ H_i=q_j\circ H_j\mbox{ on } \pi^{-1}(\mathcal{U}_i\cap\mathcal{U}_j).\]
By Fact \ref{germ def}, there exists a function $q_X$ on $T^{\ast}X$ such that 
\begin{equation}
\label{qi comptibility riemann lemma}
q_i=q_X\circ H_i^{-1}\mbox{ on } \Omega_i\times \mathbb R^d,\quad i\in\mathbb{I}. 
\end{equation}
The function $q_X$ is the square of the length function on $T^{\ast}X$ defined by the induced Riemannian metric on the cotangent bundle $T^*X$.  

\begin{defi} The function $e^{-q_X}$ on $T^{\ast}X$ is called the canonical weight of the Riemannian manifold $(X,G)$.
\end{defi}

If $X$ is compact, then $e^{-q_X}\in L_1(T^{\ast}X,\lambda).$ The functional on $C(S^{\ast}X)$ given by the formula
\[f\to\int_{T^{\ast}X}fe^{-q_X}d\lambda\]
plays a crucial role. It defines the natural measure on $S^{\ast}X.$ We note that the latter functional coincides (modulo a constant factor) with integration with respect to the kinematic density on $S^{\ast}X$ (see p.318 in \cite{Chavel}).

\subsection{Laplace-Beltrami operator on compact Riemannian manifold}

\begin{nota}\label{lb nota} Let $\Omega\subset\mathbb{R}^d$ be connected and open. Let $g:\Omega\to\mathrm{GL}^+(d,\mathbb{R})$ be a smooth mapping. Laplace-Beltrami operator $\Delta_g:C^{\infty}_c(\Omega)\to C^{\infty}_c(\Omega)$ is defined by the formula
\[\Delta_g=M_{\mathrm{det}(g)^{-\frac12}}\sum_{k,l=1}^dD_kM_{\mathrm{det}(g)^{\frac12}\cdot (g^{-1})_{k,l}}D_l.\]
\end{nota}

\begin{defi}\label{good partition of unity} Let $(\phi_n)_{n=1}^N\subset C^{\infty}(X)$ be a finite partition of unity. We call it good if each $\phi_n$ is compactly supported in some chart.
\end{defi}

Obviously, good partitions of unity exist only on compact manifolds. 

\begin{defi}\label{lb manifold def} Let $(X,G)$ be a compact Riemannian manifold. Let $\Omega_i$ be as in Notation \ref{manifold nota} and let $g_i:\Omega_i\to \mathrm{GL}^+(d,\mathbb{R})$ be as in Notation \ref{qi nota}.  Let $\Delta_{g_i}:C^{\infty}_c(\Omega_i)\to C^{\infty}_c(\Omega_i),$ $i\in\mathbb{I},$ be the Laplace-Beltrami operator as in Notation \ref{lb nota}. Let $(\phi_n)_{n=1}^N$ be a good partition of unity.
	
Laplace-Beltrami operator $\Delta_G:C^{\infty}(X)\to C^{\infty}(X)$ is defined by the formula
\[\Delta_Gf=\sum_{n=1}^N\Big(\Delta_{g_{i_n}}\big((f\phi_n)\circ h_{i_n}^{-1}\big)\Big)\circ h_{i_n},\quad f\in C^{\infty}(X).\]
Here, $i_n\in\mathbb{I}$ is chosen such that $\phi_n$ is compactly supported in $\mathcal{U}_{i_n}.$
\end{defi}

Though Definition \ref{lb manifold def} involves good partition of unity, the operator $\Delta_G$ does not actually depend on the particular choice of a good partition of unity.

Theorem 2.4 in \cite{Strichartz} yields the following results. The first one is of conceptual importance. The second one is used in the proof of Theorem \ref{ctt manifold}.

\begin{thm} Let $(X,G)$ be a compact Riemannian manifold. Laplace-Beltrami operator admits a self-adjoint extension $\Delta_G:W^{2,2}(X)\to L_2(X).$
\end{thm}

\begin{thm}\label{lb rd self-adjoint} Let $g:\mathbb{R}^d\to\mathrm{GL}^+(d,\mathbb{R}).$ Suppose that
\begin{enumerate}
\item $g\in C^{\infty}(\mathbb{R}^d,M_d(\mathbb{C}))$ (that is, $g$ is smooth and all derivatives are bounded);
\item $\mathrm{det}(g)\geq c$ for some $0<c\in\mathbb{R};$
\end{enumerate}
The operator $\Delta_g:W^{2,2}(\mathbb{R}^d)\to L_2(\mathbb{R}^d)$ is self-adjoint.
\end{thm}

\subsection{Density on a manifold}

Let $\mathfrak{B}$ be the Borel $\sigma$-algebra on the manifold $X.$ We need the notion of density on a manifold available, e.g. in \cite[p.87]{michor}.

\begin{defi}\label{condition on nu} Let $\nu$ be a countably additive measure on $\mathfrak{B}$. We assume that, for every $i\in\mathbb{I},$ the measure $\nu\circ h_i^{-1}$ on $\Omega_i$ is absolutely continuous with respect to the Lebesgue measure on $\Omega_i,$ and its Radon-Nikodym derivative $a_i$ is strictly positive and continuous on $\Omega_i$. 
	
In this case, we say that $\nu$ is a continuous positive density on $X.$
\end{defi}

\section{Invariance of principal symbol under diffeomorphisms}\label{diff invar section}

In this section, we formulate a theorem which provides a partial positive answer to Question \ref{higson question}. This result is stated in two versions: Theorem \ref{algebra invariance special} for diffeomorphisms of $\mathbb{R}^d$ (which a core technical difficulty) and Theorem \ref{algebra invariance theorem} for local diffeomorphisms (the result which would be actually used).

\subsection{Invariance under diffeomorphisms of $\mathbb{R}^d$}

We need the following notations. Recall that $\mathrm{GL}(d,\mathbb{R})$ stands for the group of inverible real $d\times d$ matrices.  

\begin{nota}\label{s2 notation} Let $\Phi:\mathbb{R}^d\to\mathbb{R}^d$ be a diffeomorphism.
\begin{enumerate}
\item Let $J_{\Phi}:\mathbb{R}^d\to\mathrm{GL}(d,\mathbb{R})$ be the Jacobian matrix of $\Phi;$
\item Let unitary operator $U:L_2(\mathbb{R}^d)\to L_2(\mathbb{R}^d)$ be defined by setting
\[(U_{\Phi}\xi)(t)=|\mathrm{det}(J_{\Phi})|^{\frac12}(t)\xi(\Phi(t)),\quad \xi\in L_2(\mathbb{R}^d),\quad t\in \mathbb{R}^d.\]
\item Let $\Theta_{\Phi}:\mathbb{R}^d\times\mathbb{S}^{d-1}\to \mathbb{R}^d\times\mathbb{S}^{d-1}$ be defined by setting
\[\Theta_{\Phi}(t,s)=(\Phi^{-1}(t),O_{J_{\Phi}^{\ast}(\Phi^{-1}(t))}s),\quad t\in \mathbb{R}^d,\quad s\in\mathbb{S}^{d-1},\]
where $J_{\Phi}^{\ast}$ is the adjoint to the Jacobi matrix.

Here, for $A\in \mathrm{GL}(d,\mathbb{R}),$ we set
\[O_As=\frac{As}{|As|},\quad s\in\mathbb{S}^{d-1}.\]
\item Let $\Xi_{\Phi}:\mathbb{R}^d\times\mathbb{R}^d\to\mathbb{R}^d\times\mathbb{R}^d$ be defined by setting
\[\Xi_{\Phi}(t,s)=(\Phi^{-1}(t),J_{\Phi}^{\ast}(\Phi^{-1}(t))s),\quad t,s\in \mathbb{R}^d.\]
\end{enumerate}
\end{nota}

We frequently need the following compatibility lemma.

\begin{lem}\label{theta composition lemma} Let $\Phi_1,\Phi_2:\mathbb{R}^d\to\mathbb{R}^d$ be diffeomorphisms. We have
\[\Theta_{\Phi_1\circ\Phi_2}=\Theta_{\Phi_2}\circ\Theta_{\Phi_1}.\]
\end{lem}
\begin{proof} Indeed,
\[\Theta_{\Phi}(t,s)=(\Phi^{-1}(t),O_{A^{\Phi}(t)}s),\]
where
\[A^{\Phi}(t)=J_{\Phi}^{\ast}(\Phi^{-1}(t)).\]
	
We have
\[(\Theta_{\Phi_2}\circ\Theta_{\Phi_1})(t,s)=\Theta_{\Phi_2}(t',s'),\quad (t',s')=(\Phi_1^{-1}(t),O_{A^{\Phi_1}(t)}s).\]
Thus,
\[(\Theta_{\Phi_2}\circ\Theta_{\Phi_1})(t,s)=(\Phi_2^{-1}(t'),O_{A^{\Phi_2}(t')}s')=(\Phi_2^{-1}(\Phi_1^{-1}(t)), O_{A^{\Phi_2}(\Phi_1^{-1}(t))}O_{A^{\Phi_1}(t)}s).\]
Note that
\[(O_{A_1}\circ O_{A_2})s=\frac{A_1(O_{A_2}s)}{|A_1(O_{A_2}s)|}=\frac{\frac{A_1A_2s}{|A_2s|}}{\frac{|A_1A_2s|}{|A_2s|}}=\frac{A_1A_2s}{|A_1A_2s|}=O_{A_1\cdot A_2}s,\quad s\in \mathbb{S}^{d-1}.\]
Since $O_{A_1}\circ O_{A_2}=O_{A_1\cdot A_2},$ it follows that
\[(\Theta_{\Phi_2}\circ\Theta_{\Phi_1})(t,s)=((\Phi_1\circ\Phi_2)^{-1}(t), O_{A^{\Phi_2}(\Phi_1^{-1}(t))\cdot A^{\Phi_1}(t)}s).\]
At the same time,
\[\Theta_{\Phi_1\circ\Phi_2}=((\Phi_1\circ\Phi_2)^{-1}(t), O_{A^{\Phi_1\circ\Phi_2}(t)}s).\]
It suffices to show that
\[A^{\Phi_1\circ\Phi_2}(t)=A^{\Phi_2}(\Phi_1^{-1}(t)) \cdot A^{\Phi_1}(t).\]
The latter equality is written as
\[J_{\Phi_1\circ\Phi_2}^{\ast}((\Phi_1\circ\Phi_2)^{-1}(t))=J_{\Phi_2}^{\ast}(\Phi_2^{-1}(\Phi_1^{-1}(t))) \cdot J_{\Phi_1}^{\ast}(\Phi_1^{-1}(t)).\]
Replacing $t$ with $(\Phi_1\circ\Phi_2)(t),$ we need to verify that
\[J_{\Phi_1\circ\Phi_2}^{\ast}(t)=J_{\Phi_2}^{\ast}(t) \cdot J_{\Phi_1}^{\ast}(\Phi_2(t)).\]
In other words,
\[J_{\Phi_1\circ\Phi_2}(t)=J_{\Phi_1}(\Phi_2(t))\cdot J_{\Phi_2}(t).\]
This is the chain rule property.	
\end{proof}

\begin{cor} $\Theta_{\Phi}:\mathbb{R}^d\times\mathbb{S}^{d-1}\to \mathbb{R}^d\times\mathbb{S}^{d-1}$ is a diffeomorphism.
\end{cor}
\begin{proof} Obviously, $J_{\Phi}^{\ast}\circ\Phi^{-1}:\mathbb{R}^d\to \mathrm{GL}(d,\mathbb{R})$ is a smooth mapping. For every smooth mapping $A:\mathbb{R}^d\to\mathrm{GL}(d,\mathbb{R}),$ the mapping
\[(t,s)\to O_{A(t)}s,\quad t\in\mathbb{R}^d,\quad s\in\mathbb{S}^{d-1},\]
is smooth. Thus, $\Theta_{\Phi}$ is smooth.

By Lemma \ref{theta composition lemma}, its inverse is $\Theta_{\Phi^{-1}}$ which is also a smooth mapping.
\end{proof}

We are now ready to state the main result in this subsection.

\begin{defi}\label{comp supp def} We say that $T\in\Pi$ is compactly supported if there exists $\phi\in C^{\infty}_c(\mathbb{R}^d)$ such that $T=T\pi_1(\phi)=\pi_1(\phi)T.$ 
\end{defi}

\begin{thm}\label{algebra invariance special} Let $\Phi:\mathbb{R}^d\to\mathbb{R}^d$ be a diffeomorphism such that $\Phi$ is affine outside of some ball. If $T\in \Pi$ is compactly supported, then $U_{\Phi}^{-1}TU_{\Phi}\in \Pi.$ Furthermore,
\[\mathrm{sym}\Big(U_{\Phi}^{-1}TU_{\Phi}\Big)=\mathrm{sym}(T)\circ\Theta_{\Phi}.\]
If we view symbols as homogeneous functions\footnote{A function $f:\mathbb{R}^d\times\mathbb{S}^{d-1}$ can be uniquely extended to a homogeneous function on $\mathbb{R}^d\times(\mathbb{R}^d\backslash\{0\})$ by setting
\[f(t,s)\stackrel{def}{=}f(t,\frac{s}{|s|}),\quad t,s\in\mathbb{R}^d.\]

If $f$ is homogeneous, then
\[(f\circ\Theta_{\Phi})(t,s)=f(\Phi^{-1}(t),\frac{J_{\Phi}^{\ast}(\Phi^{-1}(t))s}{|J_{\Phi}^{\ast}(\Phi^{-1}(t))s|})=f(\Phi^{-1}(t),J_{\Phi}^{\ast}(\Phi^{-1}(t))s)=(f\circ\Xi_{\Phi})(t,s).\]
} on $\mathbb{R}^d\times(\mathbb{R}^d\backslash\{0\}),$ then
\[\mathrm{sym}\Big(U_{\Phi}^{-1}TU_{\Phi}\Big)=\mathrm{sym}(T)\circ \Xi_{\Phi}.\]
\end{thm}

We prove Theorem \ref{algebra invariance special} in Section \ref{ais section}.

There are two reasons for us to require that $\Phi$ is affine outside of some ball. The first reason is that having equivariant behavior of the principal symbol under such diffeomorphisms is sufficient in the proof of Theorem \ref{algebra invariance theorem} below. The second reason is that in the proof of Theorem \ref{algebra invariance special} we conjugate the Laplacian with $U_{\Phi}.$ Hence, it is of crucial importance that $U_{\Phi}$ preserves the domain of Laplacian. Recall that the domain of Laplacian is a Sobolev spaces $W^{2,2}(\mathbb{R}^d)$ and, by Theorem \ref{322 taylor}, $U_{\Phi}$ leaves the domain of Laplacian invariant.

\subsection{Invariance under local diffeomorphisms}

One may ask how does the algebra and the principal symbol mapping (locally) behaves under the change of coordinates. We need the following notations.

\begin{nota}\label{s3 first nota} Let $H$ be a Hilbert space and let $p\in B(H)$ be a projection.
\begin{enumerate}
\item If $T\in B(H)$ is such that $T=pTp,$ then we define the operator $\mathrm{Rest}_p(T)\in B(pH)$ by setting $\mathrm{Rest}_p(H)=T|_{pH}.$
\item If $T\in B(pH),$ then we define $\mathrm{Ext}_p(T)\in B(H)$ by setting $\mathrm{Ext}_p(T)=T\circ p.$
\end{enumerate}
\end{nota}

\begin{nota}\label{s3 second nota} Let $\Omega\subset\mathbb{R}^d.$ If $T\in B(L_2(\mathbb{R}^d))$ is such that $T=M_{\chi_{\Omega}}TM_{\chi_{\Omega}},$ then $\mathrm{Rest}_{\Omega}(T)\in B(L_2(\Omega))$ is a shorthand for $\mathrm{Rest}_{M_{\chi_{\Omega}}}(T).$ If $T\in B(L_2(\Omega)),$ then $\mathrm{Ext}_{\Omega}(T)\in B(L_2(\mathbb{R}^d))$ is a shorthand for $\mathrm{Ext}_{M_{\chi_{\Omega}}}(T).$
\end{nota}

\begin{nota}\label{s3 notation} Let $\Omega,\Omega'\subset\mathbb{R}^d$ be open sets and let $\Phi:\Omega\to\Omega'$ be a diffeomorphism.
\begin{enumerate}
\item Let $J_{\Phi}:\Omega\to \mathrm{GL}(d,\mathbb{R})$ be the Jacobian matrix of $\Phi;$
\item Let unitary operator $U_{\Phi}:L_2(\Omega')\to L_2(\Omega)$ be defined by setting
\[(U_{\Phi}\xi)(t)=|\mathrm{det}(J_{\Phi})|^{\frac12}(t)\xi(\Phi(t)),\quad \xi\in L_2(\Omega'),\quad t\in \Omega;\]
\item Let $\Theta_{\Phi}:\Omega'\times\mathbb{S}^{d-1}\to \Omega\times\mathbb{S}^{d-1}$ be defined by setting
\[\Theta_{\Phi}(t,s)=(\Phi^{-1}(t),O_{J_{\Phi}^{\ast}(\Phi^{-1}(t))}s),\quad t\in \Omega',\quad s\in\mathbb{S}^{d-1}.\]
\item Let $\Xi_{\Phi}:\Omega'\times\mathbb{R}^d\to\Omega\times\mathbb{R}^d$ be defined by setting
\[\Xi_{\Phi}(t,s)=(\Phi^{-1}(t),J_{\Phi}^{\ast}(\Phi^{-1}(t))s),\quad t\in\Omega',\quad s\in \mathbb{R}^d.\]
\end{enumerate}
\end{nota}

\begin{lem}\label{theta domain composition lemma} Let $\Phi_1:\Omega\to\Omega'$ and $\Phi_2:\Omega''\to\Omega$ be diffeomorphisms. We have
\[\Theta_{\Phi_1\circ\Phi_2}=\Theta_{\Phi_2}\circ\Theta_{\Phi_1}.\]
\end{lem}
\begin{proof} The proof is identical to that of Lemma \ref{theta composition lemma}.
\end{proof}

\begin{cor} $\Theta_{\Phi}:\Omega'\times\mathbb{S}^{d-1}\to \Omega\times\mathbb{S}^{d-1}$ is a diffeomorphism.
\end{cor}
\begin{proof} Obviously, $J_{\Phi}^{\ast}\circ\Phi^{-1}:\Omega'\to \mathrm{GL}(d,\mathbb{R})$ is a smooth mapping. For every smooth mapping $A:\Omega\to\mathrm{GL}(d,\mathbb{R}),$ the mapping
\[(t,s)\to O_{A(t)}s,\quad t\in\Omega,\quad s\in\mathbb{S}^{d-1},\]
is smooth. Thus, $\Theta_{\Phi}$ is smooth. By Lemma \ref{theta domain composition lemma}, its inverse is $\Theta_{\Phi^{-1}}$ which is also a smooth mapping.
\end{proof}

\begin{thm}\label{algebra invariance theorem} Let $\Omega,\Omega'\subset\mathbb{R}^d$ be open sets and let $\Phi:\Omega\to\Omega '$ be a diffeomorphism. If $T\in\Pi$ is compactly supported in $\Omega,$ then
\[\mathrm{Ext}_{\Omega'}\Big(U_{\Phi}^{-1}\cdot \mathrm {Rest}_{\Omega}(T)\cdot U_{\Phi}\Big)\in\Pi.\]
Furthermore,
\[\mathrm{sym}\Big(\mathrm{Ext}_{\Omega'}\Big(U_{\Phi}^{-1}\cdot \mathrm{Rest}_{\Omega}(T)\cdot U_{\Phi}\Big)\Big)=\mathrm{sym}(T)\circ\Theta_{\Phi}.\]
If we view symbols as homogeneous functions on $\mathbb{R}^d\times(\mathbb{R}^d\backslash\{0\}),$ then
\[\mathrm{sym}\Big(\mathrm{Ext}_{\Omega'}\Big(U_{\Phi}^{-1}\cdot \mathrm{Rest}_{\Omega}(T)\cdot U_{\Phi}\Big)\Big)=\mathrm{sym}(T)\circ \Xi_{\Phi}.\]
\end{thm}

Theorem \ref{algebra invariance theorem} is proved in Section \ref{ait section} as a corollary of Theorem \ref{algebra invariance special}.

\section{Conjugation of differential operators with $U_{\Phi}$}

In this section, we examine the operators $U_{\Phi}^{-1}D_kU_{\Phi}:W^{1,2}(\mathbb{R}^d)\to L_2(\mathbb{R}^d)$ and $U_{\Phi}^{-1}\Delta U_{\Phi}:W^{2,2}(\mathbb{R}^d)\to L_2(\mathbb{R}^d)$ and show that they may be viewed as differential operators.

\begin{lem}\label{vphi bdd lemma} Let $\Phi:\mathbb{R}^d\to\mathbb{R}^d$ be a diffeomorphism such that $\Phi$ is affine outside of some ball. Mapping $V_{\Phi}$ by setting $V_{\Phi}:\xi\to\xi\circ\Phi$ is bounded on $L_2(\mathbb{R}^d)$ and so is $V_{\Phi}^{-1}.$
\end{lem}
\begin{proof} This is a special case of Theorem \ref{322 taylor} (or it can be verified by hands).
\end{proof}

\begin{lem}\label{pi1u lemma} Let $\Phi:\mathbb{R}^d\to\mathbb{R}^d$ be a diffeomorphism such that $\Phi$ is affine outside of some ball. We have
\[U_{\Phi}^{-1}\pi_1(f)U_{\Phi}=V_{\Phi}^{-1}\pi_1(f)V_{\Phi}=\pi_1(f\circ\Phi^{-1}).\]
\end{lem}
\begin{proof} By definition of $U_{\Phi}$ (in Notation \ref{s2 notation}) and of $V_{\Phi}$ (in Lemma \ref{vphi bdd lemma}), we have
\[U_{\Phi}=M_{|\mathrm{det}(J_{\Phi})|^{\frac12}}V_{\Phi}.\]
It is immediate that
\[U_{\Phi}^{-1}M_fU_{\Phi}=V_{\Phi}^{-1}M_{|\mathrm{det}(J_{\Phi})|^{-\frac12}}M_fM_{|\mathrm{det}(J_{\Phi})|^{\frac12}}V_{\Phi}.\]
Since
\[M_{|\mathrm{det}(J_{\Phi})|^{-\frac12}}M_fM_{|\mathrm{det}(J_{\Phi})|^{\frac12}}=M_f,\]
it follows that
\[U_{\Phi}^{-1}M_fU_{\Phi}=V_{\Phi}^{-1}M_fV_{\Phi}=M_{f\circ\Phi^{-1}}.\]
The assertion of the lemma now follows from the Definition \ref{pis def}.
\end{proof}

\begin{nota}\label{phi conjugate notation} Let $\Phi:\mathbb{R}^d\to\mathbb{R}^d$ be a diffeomorphism. Denote
\[(a_{k,l}^{\Phi})_{k,l=1}^d=J_{\Phi}^{\ast}\circ\Phi^{-1},\quad (b_{k,l}^{\Phi})_{k,l=1}^d=|J_{\Phi}^{\ast}\circ\Phi^{-1}|^2.\]
\end{nota}

\begin{lem}\label{musor} Functions
\[a_k^{\Phi}=\Big(|\mathrm{det}(J_{\Phi})|^{-\frac12}\cdot D_k(|\mathrm{det}(J_{\Phi})|^{\frac12})\Big)\circ\Phi^{-1},\quad 1\leq k\leq d,\]
\[b_l^{\Phi}=2\sum_{k=1}^d\Re(\bar{a}_k^{\Phi}\cdot a_{k,l}^{\Phi}),\quad 1\leq l\leq d,\]
\[b^{\Phi}=\sum_{k=1}^d\sum_{l=1}^dD_l(\bar{a}_{k,l}^{\Phi}\cdot a_k^{\Phi})+\sum_{k=1}^d|a_k^{\Phi}|^2,\]
belong to $C^{\infty}_c(\mathbb{R}^d).$
\end{lem}
\begin{proof} Since $\Phi$ is a diffeomorphism, it follows that all those functions are smooth. Since $\Phi$ is affine outside of some ball, it follows that $J_{\Phi}$ is constant outside of some ball. Thus, $D_k(|\mathrm{det}(J_{\Phi})|^{\frac12})=0$ outside of some ball. Using the definition of $a_k^{\Phi},$ we now see that it vanishes outside of some ball. Using the definition of $b_l^{\Phi}$ and $b^{\Phi},$ we now see that it vanishes outside of some ball.
\end{proof}

\begin{lem}\label{vphi conjugation explicit lemma} Let $\Phi:\mathbb{R}^d\to\mathbb{R}^d$ be a diffeomorphism such that $\Phi$ is affine outside of some ball. We have
\[V_{\Phi}^{-1}D_kV_{\Phi}=\sum_{l=1}^dM_{a_{k,l}^{\Phi}}D_l,\quad 1\leq k\leq d,\]
Here, equalities are understood as equalities of differential operators acting from $W^{1,2}(\mathbb{R}^d)$ to $L_2(\mathbb{R}^d).$
\end{lem}
\begin{proof} By the chain rule, we have 
\[D_kV_{\Phi}\xi=D_k(\xi\circ\Phi)=\sum_{l=1}^d((D_l\xi)\circ\Phi)\cdot iD_k\Phi_l,\quad \xi\in W^{1,2}(\mathbb{R}^d).\]
Using the notations for $V_{\Phi}$ (in Lemma \ref{vphi bdd lemma}) and for the multiplication operator, we can rewrite this formula as follows: 
\[D_kV_{\Phi}=\sum_{l=1}^dM_{iD_k\Phi_l}V_{\Phi}D_l.\]
Thus, 
\[V_{\Phi}^{-1}D_kV_{\Phi}=\sum_{l=1}^dV_{\Phi}^{-1}M_{iD_k\Phi_l}V_{\Phi}\cdot D_l\stackrel{L.\ref{pi1u lemma}}{=}\sum_{l=1}^dM_{i(D_k\Phi_l)\circ\Phi^{-1}}D_l=\sum_{l=1}^dM_{a_{k,l}^{\Phi}}D_l,\]
where the last equality follows from the definition of $a_{k,l}^{\Phi}$ (in Notation \ref{phi conjugate notation}) and the fact that $J_{\Phi}=(iD_l\Phi_k)_{k,l=1}^d.$
\end{proof}

\begin{lem}\label{sobolev uphi invariance lemma} Let $\Phi:\mathbb{R}^d\to\mathbb{R}^d$ be a diffeomorphism such that $\Phi$ is affine outside of some ball. We have $U_{\Phi}:W^{m,2}(\mathbb{R}^d)\to W^{m,2}(\mathbb{R}^d),$ $m\in\mathbb{Z}_+.$
\end{lem}
\begin{proof} By definition of $U_{\Phi}$ (in Notation \ref{s2 notation}) and of $V_{\Phi}$ (in Lemma \ref{vphi bdd lemma}), we have
\[U_{\Phi}=M_{h_{\Phi}}V_{\Phi},\quad h_{\Phi}=|\mathrm{det}(J_{\Phi})|^{\frac12}.\]
By Theorem \ref{322 taylor}, $V_{\Phi}:W^{m,2}(\mathbb{R}^d)\to W^{m,2}(\mathbb{R}^d).$ Since $\Phi$ is a diffeomorphism and since $\Phi$ is affine outside of some ball, it follows that $h_{\Phi}$ is a smooth function on $\mathbb{R}^d$ which is constant outside of some ball. It follows that $M_{h_{\Phi}}:W^{m,2}(\mathbb{R}^d)\to W^{m,2}(\mathbb{R}^d).$ A combination of those mappings yields the assertion.
\end{proof}

By Lemma \ref{sobolev uphi invariance lemma}, we have
\[W^{1,2}(\mathbb{R}^d)\stackrel{U_{\Phi}}{\rightarrow}W^{1,2}(\mathbb{R}^d)\stackrel{D_k}{\rightarrow}L_2(\mathbb{R}^d)\stackrel{U_{\Phi}^{-1}}{\rightarrow}L_2(\mathbb{R}^d),\]
\[W^{2,2}(\mathbb{R}^d)\stackrel{U_{\Phi}}{\rightarrow}W^{2,2}(\mathbb{R}^d)\stackrel{\Delta}{\rightarrow}L_2(\mathbb{R}^d)\stackrel{U_{\Phi}^{-1}}{\rightarrow}L_2(\mathbb{R}^d).\]
Hence, we may view $U_{\Phi}^{-1}D_kU_{\Phi}$ (respectively, $U_{\Phi}^{-1}\Delta U_{\Phi}$) as operators from $W^{1,2}(\mathbb{R}^d)$ (respectively, from $W^{2,2}(\mathbb{R}^d)$) to $L_2(\mathbb{R}^d).$

\begin{lem}\label{uphi conjugation explicit lemma} Let $\Phi:\mathbb{R}^d\to\mathbb{R}^d$ be a diffeomorphism such that $\Phi$ is affine outside of some ball. We have
\begin{enumerate}
\item\label{ucela} 
\[U_{\Phi}^{-1}D_kU_{\Phi}=\sum_{l=1}^dM_{a_{k,l}^{\Phi}}D_l+M_{a_k^{\Phi}},\]
\item\label{ucelb}
\[-U_{\Phi}^{-1}\Delta U_{\Phi}=\sum_{l_1,l_2=1}^dD_{l_1}M_{b_{l_1,l_2}^{\Phi}}D_{l_2}+\sum_{l=1}^dM_{b_l^{\Phi}}D_l+M_{b^{\Phi}}.\]
\end{enumerate}
Here, equalities are understood as equalities of linear operators acting from $W^{1,2}(\mathbb{R}^d)$ (respectively, from $W^{2,2}(\mathbb{R}^d)$) to $L_2(\mathbb{R}^d).$
\end{lem}
\begin{proof} Repeating beginning of the proof of Lemma \ref{sobolev uphi invariance lemma}, we write
\[U_{\Phi}=M_{h_{\Phi}}V_{\Phi},\quad h_{\Phi}=|\mathrm{det}(J_{\Phi})|^{\frac12}.\]
It is immediate that
\begin{equation}\label{ucel eq0}
U_{\Phi}^{-1}D_kU_{\Phi}=V_{\Phi}^{-1}M_{h_{\Phi}^{-1}}D_kM_{h_{\Phi}}V_{\Phi}.
\end{equation}
Clearly,
\begin{equation}\label{ucel eq1}
M_{h_{\Phi}^{-1}}D_kM_{h_{\Phi}}=D_k+M_{h_{\Phi}^{-1}}\cdot [D_k,M_{h_{\Phi}}]=D_k+M_{h_{\Phi}^{-1}\cdot D_kh_{\Phi}}.
\end{equation}
Combining \eqref{ucel eq0} and \eqref{ucel eq1}, we obtain
\begin{equation}\label{ucel eq2}
U_{\Phi}^{-1}D_kU_{\Phi}=V_{\Phi}^{-1}D_kV_{\Phi}+V_{\Phi}^{-1}M_{h_{\Phi}^{-1}\cdot D_kh_{\Phi}}V_{\Phi}.
\end{equation}
It follows from Lemma \ref{pi1u lemma} and the definition of $a_k^{\Phi}$ (in Lemma \ref{musor}) that
\begin{equation}\label{ucel eq3}
V_{\Phi}^{-1}M_{h_{\Phi}^{-1}\cdot D_kh_{\Phi}}V_{\Phi}=M_{a_k^{\Phi}}.
\end{equation}
The equality \eqref{ucela} follows by combining Lemma \ref{vphi conjugation explicit lemma}, \eqref{ucel eq2} and \eqref{ucel eq3}.

Taking the adjoint of \eqref{ucela}, we write
\[U_{\Phi}^{-1}D_kU_{\Phi}=\sum_{l=1}^dD_lM_{\bar{a}_{k,l}^{\Phi}}+M_{\bar{a}_k^{\Phi}}.\]
Thus,
\[U_{\Phi}^{-1}D_k^2U_{\Phi}=\big(\sum_{l=1}^dD_lM_{\bar{a}_{k,l}^{\Phi}}+M_{\bar{a}_k^{\Phi}})\cdot\big(\sum_{l=1}^dM_{a_{k,l}^{\Phi}}D_l+M_{a_k^{\Phi}}\big)=\]
\[=\sum_{l_1,l_2=1}^dD_{l_1}M_{\bar{a}_{k,l_1}^{\Phi}}M_{a_{k,l_2}^{\Phi}}D_{l_2}+\sum_{l_1=1}^dD_{l_1}M_{\bar{a}_{k,l_1}^{\Phi}}M_{a_k^{\Phi}}+\]
\[+\sum_{l_2=1}^dM_{\bar{a}_k^{\Phi}}M_{a_{k,l_2}^{\Phi}}D_{l_2}+M_{|a_k^{\Phi}|^2}.\]
Clearly,
\[\sum_{l_1=1}^dD_{l_1}M_{\bar{a}_{k,l_1}^{\Phi}}M_{a_k^{\Phi}}+\sum_{l_2=1}^dM_{\bar{a}_k^{\Phi}}M_{a_{k,l_2}^{\Phi}}D_{l_2}=\sum_{l=1}^dD_lM_{\bar{a}_{k,l}^{\Phi}\cdot a_k^{\Phi}}+\sum_{l=1}^dM_{\bar{a}_k^{\Phi}\cdot a_{k,l}^{\Phi}}D_l=\]
\[=\sum_{l=1}^dM_{\bar{a}_{k,l}^{\Phi}\cdot a_k^{\Phi}}D_l+\sum_{l=1}^dM_{\bar{a}_k^{\Phi}\cdot a_{k,l}^{\Phi}}D_l+\sum_{l=1}^d[D_l,M_{\bar{a}_{k,l}^{\Phi}\cdot a_k^{\Phi}}]=\]
\[\stackrel{\eqref{dk mf commutator eq}}{=}2\sum_{l=1}^dM_{\Re(\bar{a}_k^{\Phi}\cdot a_{k,l}^{\Phi})}D_l+\sum_{l=1}^dM_{D_l(\bar{a}_{k,l}^{\Phi}\cdot a_k^{\Phi})}.\]
Thus,
\[-U_{\Phi}^{-1}\Delta U_{\Phi}=\sum_{k=1}^dU_{\Phi}^{-1}D_k^2U_{\Phi}=\sum_{k=1}^d\sum_{l_1,l_2=1}^dD_{l_1}M_{\bar{a}_{k,l_1}^{\Phi}\cdot a_{k,l_2}^{\Phi}}D_{l_2}+\]
\[+2\sum_{k=1}^d\sum_{l=1}^dM_{\Re(\bar{a}_k^{\Phi}\cdot a_{k,l}^{\Phi})}D_l+\sum_{k=1}^d\sum_{l=1}^dM_{D_l(\bar{a}_{k,l}^{\Phi}\cdot a_k^{\Phi})}+\sum_{k=1}^dM_{|a_k^{\Phi}|^2}.\]
By the definition on $b_l^{\Phi}$ and $b^{\Phi}$ (in Lemma \ref{musor}), we have
\[-U_{\Phi}^{-1}\Delta U_{\Phi}=\sum_{k=1}^d\sum_{l_1,l_2=1}^dD_{l_1}M_{\bar{a}_{k,l_1}^{\Phi}\cdot a_{k,l_2}^{\Phi}}D_{l_2}+\sum_{l=1}^dM_{b_l^{\Phi}}D_l+M_{b^{\Phi}}.\]
Consider now the highest order term. Recalling Notation \ref{phi conjugate notation}, we write
\[\sum_{k=1}^d\bar{a}_{k,l_1}^{\Phi}a_{k,l_2}^{\Phi}=\big(|(a_{k,l}^{\Phi})_{k,l=1}^d|^2\big)_{l_1,l_2}=\big(|J_{\Phi}^{\ast}\circ\Phi^{-1}|^2\big)_{l_1,l_2}=b_{l_1,l_2}^{\Phi}.\]
This delivers \eqref{ucelb}.
\end{proof}

\section{Proof of Theorem \ref{algebra invariance special}}\label{ais section}

The proof of Theorem \ref{algebra invariance special} is somewhat technical and is presented below in the series of lemmas. The strategy is as follows:
\begin{enumerate}
\item to show that every compact operator on $L_2(\mathbb{R}^d)$ belongs to $\Pi;$
\item to show that the conjugation of $M_{\phi}\frac{D_k}{\sqrt{1-\Delta}},$ $\phi\in C^{\infty}_c(\mathbb{R}^d),$ by $U_{\Phi}$ belongs to $\Pi$ modulo compact operators;
\item to conclude that the conjugation of $M_{\phi}\frac{D_k}{\sqrt{-\Delta}},$ $\phi\in C^{\infty}_c(\mathbb{R}^d),$ by $U_{\Phi}$ belongs to $\Pi;$ 
\item to conclude the argument in Theorem \ref{algebra invariance special};
\end{enumerate}

The following assertion is well-known (see e.g. Corollary 4.1.10 in \cite{Dixmier}).
 
\begin{lem}\label{from dixmier lemma} Let $\mathcal{A}$ be a $C^{\ast}$-algebra. Let $\pi:\mathcal{A}\to B(H)$ be an irreducible representation. One of the following mutually exclusive options holds:
\begin{enumerate}
\item $\pi(\mathcal{A})$ does not contain any compact operator (except for $0$);
\item $\pi(\mathcal{A})$ contains every compact operator. 
\end{enumerate}
\end{lem}

We now apply Lemma \ref{from dixmier lemma} to the $C^{\ast}$-algebra $\mathcal{A}=\Pi$ and infer that $\Pi$ contains the ideal $\mathcal{K}(L_2(\mathbb{R}^d)).$

\begin{lem}\label{compact operators lemma} The algebra $\mathcal{K}(L_2(\mathbb{R}^d))$ is contained in $\Pi$ and coincides with the kernel of the homomorphism $\mathrm{sym}.$
\end{lem}
\begin{proof}  Since $\Pi$ contains $\pi_1(\mathcal{A}_1),$ it follows (here, $X'$ denotes the commutant of the set $X\subset B(L_2(\mathbb{R}^d))$) that
\[\Pi'\subset \Big(\pi_1(\mathcal{A}_1)\Big)'=\Big(\pi_1(L_{\infty}(\mathbb{R}^d))\Big)'=\pi_1(L_{\infty}(\mathbb{R}^d)).\]
Define $g_{n,k}\in C(\mathbb{S}^{d-1}),$ $1\leq k\leq d,$ $n\in\mathbb{N},$ by setting  $g_k(s)=s_k^{\frac1{2n+1}},$ $s\in\mathbb{S}^{d-1}.$ Clearly, $\pi_2(g_{n,k})\to \mathrm{sgn}(D_k)$ as $n\to\infty$ in weak operator topology. Thus, $\mathrm{sgn}(D_k)$ belongs to the weak closure of $\pi_2(\mathcal{A}_2)$ and, hence, to the weak closure of $\Pi.$ Therefore,
\[\Pi'\subset(\mathrm{sgn}(D_k))',\quad 1\leq k\leq d.\]
Thus,
\[\Pi'\subset\Big(\bigcap_{1\leq k\leq d}(\mathrm{sgn}(D_k))'\Big)\cap \pi_1(L_{\infty}(\mathbb{R}^d)).\]

For $t\in\mathbb{R}^d,$ denote by $\check{t}_k\in\mathbb{R}^{d-1}$ the vector obtained by eliminating the $k$-th component of $t.$ If $f\in L_{\infty}(\mathbb{R}^d)$ is such that $\pi_1(f)$ commutes with $\mathrm{sgn}(D_k),$ then, for almost every $\check{t}_k\in\mathbb{R}^{d-1},$ the function $f(\check{t}_k,\cdot)$ commutes with the Hilbert transform. This easily implies that, for almost every $\check{t}_k\in\mathbb{R}^{d-1},$ the function $f(\check{t}_k,\cdot)$ is constant.  If $f\in L_{\infty}(\mathbb{R}^d)$ is such that $\pi_1(f)$ commutes with \textit{every} $\mathrm{sgn}(D_k),$  $1\leq k\leq d,$ then $f=\mathrm{const}.$ Hence, $\Pi'$ is trivial.

By Proposition II.6.1.8 in \cite{Blackadar-book}, representation $\mathrm{id}:\Pi\to B(L_2(\mathbb{R}^d))$ is irreducible.

We now demonstrate that $\Pi$ contains a non-zero compact operator. As proved above, for every non-zero $f\in C^{\infty}_c(\mathbb{R}^d),$ there exists $1\leq k\leq d$ such that $\pi_1(f)$ does not commutes with $\mathrm{sgn}(D_k)$. Since $\pi_2(g_{n,k})\to \mathrm{sgn}(D_k)$ as $n\to\infty$ in weak operator topology, it follows that $\pi_1(f)$ does not commute with $\pi_2(g_{n,k})$ for some $n,k.$ Thus, the operator $[\pi_1(f),\pi_2(g_{n,k})]$ is a non-zero compact operator, which belongs to $\Pi.$ The first assertion of the lemma follows now from Lemma \ref{from dixmier lemma}.

Let $q:B(L_2(\mathbb{R}^d))\to B(L_2(\mathbb{R}^d))/\mathcal{K}(L_2(\mathbb{R}^d))$ be the canonical quotient map. Recall (see the proof of Theorem 3.3 in \cite{DAO2}) that $\mathrm{sym}$ is constructed as a composition 
\[\mathrm{sym}=\theta^{-1}\circ q,\]
where $\theta^{-1}$ is some linear isomorphism (its definition and properties are irrelevant at the current proof). It follows that the kernel of $\mathrm{sym}$ coincides with the kernel of $q$, which is $\mathcal{K}(L_2(\mathbb{R}^d)).$
\end{proof}

\begin{nota}\label{pre irpl notation} Let $\Phi:\mathbb{R}^d\to\mathbb{R}^d$ be a diffeomorphism such that $\Phi$ is affine outside of some ball. Denote
\[p^{\Phi}(t,s)=\sum_{l_1,l_2=1}^db_{l_1,l_2}^{\Phi}(t)s_{l_1}s_{l_2},\quad r_k^{\Phi}(t,s)=\sum_{l=1}^da_{k,l}^{\Phi}(t)s_l,\quad 1\leq k\leq d.\]
Here, $(a_{k,l}^{\Phi})_{k,l=1}^d$ and $(b_{l_1,l_2}^{\Phi})_{l_1,l_2=1}^d$ are as in Notation \ref{phi conjugate notation}.
\end{nota}

The following two lemmas form the core of our computation.

\begin{lem}\label{inverse root psdo lemma} Let $\Phi:\mathbb{R}^d\to\mathbb{R}^d$ be a diffeomorphism such that $\Phi$ is affine outside of some ball. We have
\[U_{\Phi}^{-1}\frac{D_k}{\sqrt{1-\Delta}}U_{\Phi}\in \mathrm{Op}(\frac{r_k^{\Phi}}{(1+p^{\Phi})^{\frac12}})+\Psi^{-1}(\mathbb{R}^d).\]
\end{lem}
\begin{proof} Lemma \ref{uphi conjugation explicit lemma} asserts that
\begin{equation}\label{irpl eq0}
U_{\Phi}^{-1}D_kU_{\Phi}=\mathrm{Op}(r_k^{\Phi})+M_{a_k^{\Phi}}.
\end{equation}
It is immediate that
\begin{equation}\label{irpl eq1}
\mathrm{Op}(r_k^{\Phi})\in\Psi^1(\mathbb{R}^d),\quad M_{a_k^{\Phi}}\in \Psi^0(\mathbb{R}^d).
\end{equation}
	
Lemma \ref{sobolev uphi invariance lemma} yields that $-U_{\Phi}^{-1}\Delta U_{\Phi}$ is a self-adjoint positive operator with the domain $U_{\Phi}^{-1}(W^{2,2}(\mathbb{R}^d))=W^{2,2}(\mathbb{R}^d).$ Lemma \ref{uphi conjugation explicit lemma}  and \eqref{do vs psdo} yield that $-U_{\Phi}^{-1}\Delta U_{\Phi}$ is a differential operator of order $2$ with principal symbol $p^{\Phi}\geq0.$
	
By Theorem \ref{complex power thm} applied with $T=-U_{\Phi}^{-1}\Delta U_{\Phi}$ and $z=-\frac12,$ we have 
\begin{equation}\label{irpl eq2}
(1-U_{\Phi}^{-1}\Delta U_{\Phi})^{-\frac12}-\mathrm{Op}((p^{\Phi}+1)^{-\frac12})\in \Psi^{-2}(\mathbb{R}^d),
\end{equation}
\begin{equation}\label{irpl eq3}
\mathrm{Op}((p^{\Phi}+1)^{-\frac12})\in \Psi^{-1}(\mathbb{R}^d).
\end{equation}
Equations \eqref{irpl eq0}, \eqref{irpl eq1}, \eqref{irpl eq2} and \eqref{irpl eq3} yield that the operators
\[T_1=U_{\Phi}^{-1}D_kU_{\Phi},\quad T_2=U_{\Phi}^{-1}(1-\Delta)^{-\frac12}U_{\Phi}\]
satisfy the assumptions in Lemma \ref{psi product lemma}. By Lemma \ref{psi product lemma}, we have
\[U_{\Phi}^{-1}\frac{D_k}{\sqrt{1-\Delta}}U_{\Phi}=T_1T_2\in \mathrm{Op}(r_k^{\Phi}\cdot (1+p^{\Phi})^{-\frac12})+\Psi^{-1}(\mathbb{R}^d).\]
\end{proof}

In our next lemma, we approximate the operators on the left hand side with pseudo\-differential-like operators on the right hand side. The latter is defined in \eqref{psdo-like operator def} in Subsection \ref{psdo-like subsection}.

\begin{lem}\label{post irpl lemma} Let $\Phi:\mathbb{R}^d\to\mathbb{R}^d$ be a diffeomorphism such that $\Phi$ is affine outside of some ball. If $\phi\in C^{\infty}_c(\mathbb{R}^d),$ then
\[U_{\Phi}^{-1}M_{\phi}\frac{D_k}{\sqrt{1-\Delta}}U_{\Phi}\in T_{(\phi\circ\Phi^{-1}\otimes 1)\cdot q_k^{\Phi}}+\mathcal{K}(L_2(\mathbb{R}^d)),\]
where
\[q_k^{\Phi}(t,s)=(O_{J_{\Phi}^{\ast}(\Phi^{-1}(t))}s)_k,\quad t\in\mathbb{R}^d,\quad s\in\mathbb{S}^{d-1}.\]
\end{lem}
\begin{proof} For every $f\in C^{\infty}(\mathbb{R}^d)$ and for every $p\in C^{\infty}(\mathbb{R}^d\times\mathbb{R}^d),$ we have
\[M_f\cdot \mathrm{Op}(p)=\mathrm{Op}((f\otimes 1)p).\]
Also,
\[U_{\Phi}^{-1}M_{\phi}\frac{D_k}{\sqrt{1-\Delta}}U_{\Phi}=M_{\phi\otimes\Phi^{-1}}\cdot U_{\Phi}^{-1}\frac{D_k}{\sqrt{1-\Delta}}U_{\Phi}.\]
It follows now from Lemma \ref{inverse root psdo lemma} that
\begin{equation}\label{pirpl eq0}
U_{\Phi}^{-1}M_{\phi}\frac{D_k}{\sqrt{1-\Delta}}U_{\Phi}\in \mathrm{Op}((\phi\circ\Phi^{-1}\otimes 1)\frac{r_k^{\Phi}}{(1+p^{\Phi})^{\frac12}})+\Psi^{-1}(\mathbb{R}^d).
\end{equation}

Fix a function $\psi\in C^{\infty}_c(\mathbb{R}^d)$ such that $\psi=1$ near $0$ and such that $(\phi\circ\Phi^{-1})\cdot\psi=\phi\circ\Phi^{-1}.$ Set
\[e_k(t,s)=\phi(\Phi^{-1}(t))\cdot r_k^{\Phi}(t,s)\cdot (1+p^{\Phi}(t,s))^{-\frac12},\quad t,s\in\mathbb{R}^d,\]
\[f_k(t,s)=\phi(\Phi^{-1}(t))\cdot r_k^{\Phi}(t,s)\cdot (p^{\Phi}(t,s))^{-\frac12}\cdot (1-\psi(s)),\quad t,s\in\mathbb{R}^d.\]
We have $e_k-f_k=g_k\cdot h,$ where
\[g_k(t,s)=\phi(\Phi^{-1}(t))\cdot r_k^{\Phi}(t,s),\quad t,s\in\mathbb{R}^d,\]
\[h(t,s)=(1+p^{\Phi}(t,s))^{-\frac12}-(p^{\Phi}(t,s))^{-\frac12}\cdot (1-\psi(s)),\quad t,s\in\mathbb{R}^d.\]
An elementary computation shows that
\[\sup_{t,s\in\mathbb{R}^d}(1+|s|^2)^{\frac{|\beta|_1+2}{2}}|D_t^{\alpha}D_s^{\beta}h(t,s)|<\infty,\quad \alpha,\beta\in\mathbb{Z}_+^d.\]
By the Leibniz rule, we have
\[D_t^{\alpha}D_s^{\beta}(g_k\cdot h)=\sum_{0\leq\gamma\leq\alpha}\sum_{0\leq\delta\leq\beta}c(\alpha,\gamma)c(\beta,\delta)D_t^{\gamma}D_s^{\delta}g_k\cdot D_t^{\alpha-\gamma}D_s^{\beta-\delta}h.\]
This implies
\[\sup_{t,s\in\mathbb{R}^d}(1+|s|^2)^{\frac{|\beta|_1+1}{2}}|D_t^{\alpha}D_s^{\beta}(e_k-f_k)(t,s)|<\infty,\quad \alpha,\beta\in\mathbb{Z}_+^d,\]
so that $\mathrm{Op}(e_k-f_k)\in\Psi^{-1}(\mathbb{R}^d)$ (see \eqref{psim p condition}). It follows now from \eqref{pirpl eq0} that
\begin{equation}\label{pirpl eq1}
U_{\Phi}^{-1}M_{\phi}\frac{D_k}{\sqrt{1-\Delta}}U_{\Phi}-\mathrm{Op}(f_k)\in\Psi^{-1}(\mathbb{R}^d).
\end{equation}
Denote for brevity the left hand side of \eqref{pirpl eq1} by $T$ (so that $T\in\Psi^{-1}(\mathbb{R}^d)$). Due to the choice of $\psi,$ we have that $T=M_{\psi}T.$ It follows now from \eqref{compactly supported are compact} (applied with $m=-1$) that $T$ is compact. In other words, we have
\begin{equation}\label{pirpl eq4}
U_{\Phi}^{-1}M_{\phi}\frac{D_k}{\sqrt{1-\Delta}}U_{\Phi}-\mathrm{Op}(f_k)\in\mathcal{K}(L_2(\mathbb{R}^d)).
\end{equation}

Appealing to the definition of $r_k^{\Phi}$ and $p^{\Phi},$ we note that
\[r_k^{\Phi}(t,s)=(J_{\Phi}^{\ast}(\Phi^{-1}(t))s)_k,\quad p^{\Phi}(t,s)=|J_{\Phi}^{\ast}(\Phi^{-1}(t))s|^2,\quad t,s\in\mathbb{R}^d.\]
Therefore,
\[r_k^{\Phi}(t,s)\cdot (p^{\Phi}(t,s))^{-\frac12}=(O_{J_{\Phi}^{\ast}(\Phi^{-1}(t))}\frac{s}{|s|})_k,\quad t,s\in\mathbb{R}^d.\]
Thus,
\[f_k(t,s)=\phi(\Phi^{-1}(t))\cdot q_k^{\Phi}(t,\frac{s}{|s|})\cdot (1-\psi(s)),\quad t,s\in\mathbb{R}^d.\]

By Lemma \ref{82 lemma} applied with $q=(\phi\circ\Phi^{-1}\otimes1)\cdot q_k^{\Phi},$  we have 
\begin{equation}\label{pirpl eq5}
\mathrm{Op}(f_k)-T_{(\phi\circ\Phi^{-1}\otimes1)\cdot q_k^{\Phi}}\in \mathcal{K}(L_2(\mathbb{R}^d)).
\end{equation}
Combining \eqref{pirpl eq4} and \eqref{pirpl eq5}, we complete the proof.
\end{proof}

\begin{lem}\label{delta to 1-delta lemma} Let $\Phi$ be a diffeomorphism such that $\Phi$ is affine outside of some ball. If $\phi\in C^{\infty}_c(\mathbb{R}^d),$ then
\[U_{\Phi}^{-1}M_{\phi}\frac{D_k}{\sqrt{-\Delta}}U_{\Phi}\in \Pi\]
and
\[\mathrm{sym}\Big(U_{\Phi}^{-1}M_{\phi}\frac{D_k}{\sqrt{-\Delta}}U_{\Phi}\Big)=(\phi\circ\Phi^{-1}\otimes1)\cdot q_k^{\Phi}.\]
\end{lem}
\begin{proof} Applying bounded Borel function
\[t\to (1+|t|^2)\cdot (\frac{t}{|t|}-\frac{t}{\sqrt{1+|t|^2}}),\quad t\in\mathbb{R}^d,\]
to the tuple $\nabla,$ we obtain that
\[(1-\Delta)\cdot \Big(\frac{D_k}{\sqrt{-\Delta}}-\frac{D_k}{\sqrt{1-\Delta}}\Big)\in B(L_2(\mathbb{R}^d)).\]	
Recall that (see e.g. Theorem 4.1 in \cite{Simon-book}) $M_{\phi}(1-\Delta)^{-1}\in \mathcal{K}(L_2(\mathbb{R}^d)).$
Since product of bounded and compact operators is compact, it follows that
\[M_{\phi}\Big(\frac{D_k}{\sqrt{-\Delta}}-\frac{D_k}{\sqrt{1-\Delta}}\Big)\in \mathcal{K}(L_2(\mathbb{R}^d))\]
and
\[U_{\Phi}^{-1}M_{\phi}\Big(\frac{D_k}{\sqrt{-\Delta}}-\frac{D_k}{\sqrt{1-\Delta}}\Big)U_{\Phi}\in \mathcal{K}(L_2(\mathbb{R}^d)).\]

Combining with Lemma \ref{post irpl lemma}, we obtain
\[U_{\Phi}^{-1}M_{\phi}\frac{D_k}{\sqrt{-\Delta}}U_{\Phi}- T_{(\phi\circ\Phi^{-1}\otimes1)\cdot q_k^{\Phi}}\in \mathcal{K}(L_2(\mathbb{R}^d)).\]
By Lemma \ref{81 lemma}, we have 
\[T_{(\phi\circ\Phi^{-1}\otimes1)\cdot q_k^{\Phi}}\in\Pi,\quad \mathrm{sym}\Big(T_{(\phi\circ\Phi^{-1}\otimes1)\cdot q_k^{\Phi}}\Big)=(\phi\circ\Phi^{-1}\otimes1)\cdot q_k^{\Phi}.\]
The assertion follows by combining the last two equations and Lemma \ref{compact operators lemma}.
\end{proof}

\begin{lem}\label{commutator is compact} If $T_1,T_2\in\Pi,$ then
\[[T_1,T_2]\in \mathcal{K}(L_2(\mathbb{R}^d)).\]
\end{lem}
\begin{proof} Since $\mathrm{sym}$ is a $\ast$-homomorphism, it follows that
$\mathrm{sym}([T_1,T_2])=0.$
The assertion is now an immediate consequence of Lemma \ref{compact operators lemma}.
\end{proof}

\begin{lem}\label{ber request lemma1} Let $(f_k)_{k=1}^m\subset\mathcal{A}_1$ and $(g_k)_{k=1}^m\subset\mathcal{A}_2.$ We have
\[\prod_{k=1}^m\pi_1(f_k)\pi_2(g_k)\in\pi_1(\prod_{k=1}^mf_k)\pi_2(\prod_{k=1}^mg_k)+\mathcal{K}(L_2(\mathbb{R}^d)).\]
\end{lem}
\begin{proof} We prove the assertion by induction on $m.$ For $m=1,$ there is nothing to prove. So, we only have to prove the step of induction.
	
Let us prove the assertion for $m=2.$ We have
\[\pi_1(f_1)\pi_2(g_1)\pi_1(f_2)\pi_2(g_2)=[\pi_2(g_1),\pi_1(f_1f_2)]\cdot \pi_2(g_2)+\]
\[+[\pi_1(f_1),\pi_2(g_1)]\cdot \pi_1(f_2)\pi_2(g_2)+\pi_1(f_1f_2)\pi_2(g_1g_2).\]
By Lemma \ref{commutator is compact}, we have
\[[\pi_1(f_1),\pi_2(g_1)],[\pi_2(g_1),\pi_1(f_1f_2)]\in \mathcal{K}(L_2(\mathbb{R}^d)).\]
Therefore,
\[\pi_1(f_1)\pi_2(g_1)\pi_1(f_2)\pi_2(g_2)\in \pi_1(f_1f_2)\pi_2(g_1g_2)+\mathcal{K}(L_2(\mathbb{R}^d)).\]
This proves the assertion for $m=2.$

It remains to prove the step of induction. Suppose the assertion holds for $m\geq 2$ and let us prove it for $m+1.$ Clearly,
\[\prod_{k=1}^{m+1}\pi_1(f_k)\pi_2(g_k)=\pi_1(f_1)\pi_2(g_1)\cdot\prod_{k=2}^{m+1}\pi_1(f_k)\pi_2(g_k).\]
Using the inductive assumption, we obtain
\[\prod_{k=1}^{m+1}\pi_1(f_k)\pi_2(g_k)\in\pi_1(f_1)\pi_2(g_1)\cdot\pi_1(\prod_{k=2}^{m+1}f_k)\pi_2(\prod_{k=2}^{m+1}g_k)+\mathcal{K}(L_2(\mathbb{R}^d)).\]
Using the assertion for $m=2,$ we obtain
\[\pi_1(f_1)\pi_2(g_1)\cdot\pi_1(\prod_{k=2}^{m+1}f_k)\pi_2(\prod_{k=2}^{m+1}g_k)\in\pi_1(\prod_{k=1}^{m+1}f_k)\pi_2(\prod_{k=1}^{m+1}g_k)+\mathcal{K}(L_2(\mathbb{R}^d)).\]

Combining the last two equations, we obtain
\[\prod_{k=1}^{m+1}\pi_1(f_k)\pi_2(g_k)\in\pi_1(\prod_{k=1}^{m+1}f_k)\pi_2(\prod_{k=1}^{m+1}g_k)+\mathcal{K}(L_2(\mathbb{R}^d)).\]
This establishes the step of induction and, hence, completes the proof of the lemma.
\end{proof}

\begin{lem}\label{final lemma} Let $\Phi$ be a diffeomorphism which is affine outside of some ball. If $g\in C(\mathbb{S}^{d-1})$ and $f\in C_c(\mathbb{R}^d),$ then
\[U_{\Phi}^{-1}\pi_1(f)\pi_2(g)U_{\Phi}\in \Pi\]
and
\[\mathrm{sym}\Big(U_{\Phi}^{-1}\pi_1(f)\pi_2(g)U_{\Phi}\Big)=(f\circ\Phi^{-1}\otimes1)\cdot g(q_1^{\Phi},\cdots,q_d^{\Phi}).\]
\end{lem}
\begin{proof} Let $\mathrm{Poly}(\mathbb{S}^{d-1})$ be the algebra of polynomials on $\mathbb{S}^{d-1}.$
	
Suppose first that $g\in \mathrm{Poly}(\mathbb{S}^{d-1})$ is monomial. Let $g(s)=\prod_{l=1}^ds_l^{n_l}.$ Let $\phi\in C^{\infty}_c(\mathbb{R}^d)$ be such that $f\cdot\phi=f.$ Obviously,
\[\pi_1(f)\pi_2(g)=\pi_1(f\cdot\phi^{\sum_{l=1}^dn_l})\cdot\pi_2(g)=\pi_1(f)\cdot\pi_1(\phi^{\sum_{l=1}^dn_l})\pi_2(g).\]
Setting $m=\sum_{l=1}^dn_l,$ $f_k=\phi,$ $1\leq k\leq m,$
\[g_k(s)=s_l,\quad s\in\mathbb{S}^{d-1},\quad \sum_{i=1}^{l-1}n_i<k\leq \sum_{i=1}^ln_i.\]
In this notations,
\[\pi_1(\phi^{\sum_{l=1}^dn_l})\pi_2(g)=\pi_1(\prod_{k=1}^mf_k)\pi_2(\prod_{k=1}^mg_k).\]
By Lemma \ref{ber request lemma1}, we have
\[\pi_1(\phi^{\sum_{l=1}^dn_l})\pi_2(g)\in \prod_{k=1}^m\pi_1(f_k)\pi_2(g_k)+\mathcal{K}(L_2(\mathbb{R}^d)).\]
Since
\[\pi_1(f_k)\pi_2(g_k)=\pi_1(\phi)\frac{D_l}{\sqrt{-\Delta}},\quad \sum_{i=1}^{l-1}n_i<k\leq \sum_{i=1}^ln_i,\]
it follows that
\[\pi_1(f)\pi_2(g)\in\pi_1(f)\cdot\prod_{l=1}^d\Big(\pi_1(\phi)\frac{D_l}{\sqrt{-\Delta}}\Big)^{n_l}+\mathcal{K}(L_2(\mathbb{R}^d)).\]
Thus,
\[U_{\Phi}^{-1}\pi_1(f)\pi_2(g)U_{\Phi}\in U_{\Phi}^{-1}\pi_1(f)U_{\Phi}\cdot\prod_{l=1}^d\Big(U_{\Phi}^{-1}\pi_1(\phi)\frac{D_l}{\sqrt{-\Delta}}U_{\Phi}\Big)^{n_l}+\mathcal{K}(L_2(\mathbb{R}^d)).\]

By Lemma \ref{delta to 1-delta lemma}, we have
\[U_{\Phi}^{-1}\pi_1(f)U_{\Phi}\cdot\prod_{l=1}^d\Big(U_{\Phi}^{-1}\pi_1(\phi)\frac{D_l}{\sqrt{-\Delta}}U_{\Phi}\Big)^{n_l}\in \Pi\]
and
\[\mathrm{sym}\Big(U_{\Phi}^{-1}\pi_1(f)U_{\Phi}\cdot\prod_{l=1}^d\Big(U_{\Phi}^{-1}\pi_1(\phi)\frac{D_l}{\sqrt{-\Delta}}U_{\Phi}\Big)^{n_l}\Big)=\]
\[=\mathrm{sym}\Big(U_{\Phi}^{-1}\pi_1(f)U_{\Phi}\Big)\cdot\prod_{l=1}^d\Big(\mathrm{sym}\Big(U_{\Phi}^{-1}\pi_1(\phi)\frac{D_l}{\sqrt{-\Delta}}U_{\Phi}\Big)\Big)^{n_l}=\]
\[=(f\circ\Phi^{-1}\otimes 1)\cdot\prod_{l=1}^d\Big((\phi\circ\Phi^{-1}\otimes1)\cdot q_l^{\Phi}\Big)^{n_l}=(f\circ\Phi^{-1}\otimes 1)\cdot\prod_{l=1}^d\big(q_l^{\Phi}\big)^{n_l}.\]
By Lemma \ref{compact operators lemma}, we have
\[U_{\Phi}^{-1}\pi_1(f)\pi_2(g)U_{\Phi}\in \Pi\]
and
\[\mathrm{sym}\Big(U_{\Phi}^{-1}\pi_1(f)\pi_2(g)U_{\Phi}\Big)=(f\circ\Phi^{-1}\otimes1)\cdot g(q_1^{\Phi},\cdots,q_d^{\Phi}).\]

By linearity, the same assertion holds if $g\in\mathrm{Poly}(\mathbb{S}^{d-1}).$ To prove the assertion in general, let $g\in C(\mathbb{S}^{d-1})$ and consider a sequence $\{g_n\}_{n\geq1}\subset\mathrm{Poly}(\mathbb{S}^{d-1})$ such that $g_n\to g$ in the uniform norm. We have
\[U_{\Phi}^{-1}\pi_1(f)\pi_2(g_n)U_{\Phi}\to U_{\Phi}^{-1}\pi_1(f)\pi_2(g)U_{\Phi}\]
in the uniform norm. Since
\[U_{\Phi}^{-1}\pi_1(f)\pi_2(g_n)U_{\Phi}\in \Pi,\quad n\geq1,\]
it follows that
\[U_{\Phi}^{-1}\pi_1(f)\pi_2(g)U_{\Phi}\in \Pi\]
and
\[\mathrm{sym}(U_{\Phi}^{-1}\pi_1(f)\pi_2(g_n)U_{\Phi})\to \mathrm{sym}(U_{\Phi}^{-1}\pi_1(f)\pi_2(g)U_{\Phi})\]
in the uniform norm. In other words,
\[(f\circ\Phi^{-1}\otimes1)\cdot g_n(q_1^{\Phi},\cdots,q_d^{\Phi})\to\mathrm{sym}(U_{\Phi}^{-1}\pi_1(f)\pi_2(g)U_{\Phi})\]
in the uniform norm. Thus,
\[\mathrm{sym}(U_{\Phi}^{-1}\pi_1(f)\pi_2(g)U_{\Phi})=(f\circ\Phi^{-1}\otimes1)\cdot g(q_1^{\Phi},\cdots,q_d^{\Phi}).\]
\end{proof}

\begin{proof}[Proof of Theorem \ref{algebra invariance special}] By the definition of the $C^{\ast}$-algebra $\Pi,$ for every $T\in\Pi,$ there exists a sequence $(T_n)_{n\geq1}$ in the $\ast$-algebra generated by $\pi_1(\mathcal{A}_1)$ and $\pi_2(\mathcal{A}_2)$ such that $T_n\to T$ in the uniform norm. We can write
\[T_n=\sum_{l=1}^{l_n}\prod_{k=1}^{k_n}\pi_1(f_{n,k,l})\pi_2(g_{n,k,l}).\]
By Lemma \ref{ber request lemma1}, we have
\[T_n\in\sum_{l=1}^{l_n}\pi_1(\prod_{k=1}^{k_n}f_{n,k,l})\pi_2(\prod_{k=1}^{k_n}g_{n,k,l})+\mathcal{K}(L_2(\mathbb{R}^d)).\]
Denote for brevity,
\[f_{n,l}=\prod_{k=1}^{k_n}f_{n,k,l}\in\mathcal{A}_1,\quad g_{n,l}=\prod_{k=1}^{k_n}g_{n,k,l}\in\mathcal{A}_2.\]
We have
\[T_n=S_n+\sum_{l=1}^{l_n}\pi_1(f_{n,l})\pi_2(g_{n,l}),\quad S_n\in\mathcal{K}(L_2(\mathbb{R}^d)).\]
By Lemma \ref{compact operators lemma}, we have
\begin{equation}\label{tn symbol eq}
\mathrm{sym}(T_n)=\sum_{l=1}^{l_n}f_{n,l}\otimes g_{n,l}.
\end{equation}

Suppose in addition that $T$ is compactly supported. In particular, $T=M_{\phi}T$ for some $\phi\in C^{\infty}_c(\mathbb{R}^d).$ Replacing $S_n$ with $M_{\phi}S_n$ and $f_{n,l}$ with $\phi\cdot f_{n,l}$ if necessary, we may assume without loss of generality that $f_{n,l}\in C^{\infty}_c(\mathbb{R}^d)$ for every $n\geq1$ and for every $1\leq l\leq l_n.$ 
	
By Lemma \ref{final lemma}, we have
\[\sum_{l=1}^{l_n}U_{\Phi}^{-1}\pi_1(f_{n,l})\pi_2(g_{n,l})U_{\Phi}\in \Pi\]
and
\[\mathrm{sym}\Big(\sum_{l=1}^{l_n}U_{\Phi}^{-1}\pi_1(f_{n,l})\pi_2(g_{n,l})U_{\Phi}\Big)=(\sum_{l=1}^{l_n}f_{n,l}\otimes g_{n,l})\circ\Theta_{\Phi},\]
where $\Theta_{\Phi}$ is introduced in Notation \ref{s2 notation}.

By Lemma \ref{compact operators lemma}, we have $U_{\Phi}^{-1}S_nU_{\Phi}\in \Pi$
and
\[\mathrm{sym}(U_{\Phi}^{-1}S_nU_{\Phi})=0.\]
Thus, $U_{\Phi}^{-1}T_nU_{\Phi}\in \Pi$ and
\[\mathrm{sym}\Big(U_{\Phi}^{-1}T_nU_{\Phi}\Big)=(\sum_{l=1}^{l_n}f_{n,l}\otimes g_{n,l})\circ\Theta_{\Phi}=\mathrm{sym}(T_n)\circ\Theta_{\Phi}.\]

Since $\Pi$ is a $C^{\ast}$-algebra and since $U_{\Phi}^{-1}T_nU_{\Phi}\to U_{\Phi}^{-1}TU_{\Phi}$ in the uniform norm, it follows that $U_{\Phi}^{-1}TU_{\Phi}\in \Pi$ and
\[\mathrm{sym}\Big(U_{\Phi}^{-1}T_nU_{\Phi}\Big)\to \mathrm{sym}\Big(U_{\Phi}^{-1}TU_{\Phi}\Big)\]
in the uniform norm. In other words, 
\[\mathrm{sym}(T_n)\circ\Theta_{\Phi}\to \mathrm{sym}\Big(U_{\Phi}^{-1}TU_{\Phi}\Big)\]
in the uniform norm. Since $\mathrm{sym}(T_n)\to\mathrm{sym}(T)$ in the uniform norm, it follows that
\[\mathrm{sym}\Big(U_{\Phi}^{-1}TU_{\Phi}\Big)=\mathrm{sym}(T)\circ\Theta_{\Phi}.\]
\end{proof}

\section{Invariance of principal symbol under local diffeomorphisms}\label{ait section}

Theorem \ref{algebra invariance theorem} is supposed to be a corollary of Theorem \ref{algebra invariance special}. To demonstrate this is indeed the case, we need an extension result for diffeomorphisms.

The following fundamental result is due to Palais \cite{PalaisTAMS} (see Corollary 4.3 there).

\begin{thm}\label{palais diffeomorphism theorem} Let $\Phi:\mathbb{R}^d\to\mathbb{R}^d$ be a smooth mapping. Necessary and sufficient conditions for $\Phi$ to be a  diffeomorphism are as follows:
\begin{enumerate}
\item for every $t\in\mathbb{R}^d,$ we have $\mathrm{det}(J_{\Phi}(t))\neq0$;
\item we have $\Phi(t)\to\infty$ as $|t|\to\infty;$  
\end{enumerate}
\end{thm}

The next lemma is also due to Palais \cite{PalaisTAMS}. We provide a proof for convenience of the reader. Note that $B(t,r)$ is the open ball with radius $r$ centered at $t.$ 

\begin{lem}\label{palais lemma} Let $\Omega\subset\mathbb{R}^d$ be an open set and let $\Phi:\Omega\to\mathbb{R}^d$ be a smooth mapping. If $t\in\Omega$ is such that $\mathrm{det}(J_{\Phi}(t))\neq0,$ then there exists a diffeomorphism $\Phi_t:\mathbb{R}^d\to\mathbb{R}^d$ such that
\begin{enumerate}
\item $\Phi_t=\Phi$ on $B(t,r_1(t))$ with some $r_1(t)>0;$
\item $\Phi_t$ is affine outside $B(t,r_2(t))$ for some $r_2(t)<\infty;$ 
\end{enumerate}	
\end{lem}
\begin{proof} Without loss of generality, $t=0,$ $\Phi(0)=0$ and $J_{\Phi}(0)=1_{M_d(\mathbb{R})},$ the unity in the algebra of real $d\times d$ matrices.  Let $\theta\in C^{\infty}_c(\mathbb{R}^d)$ be such that $\theta=1$ on the unit ball. Set 
\[\Psi_r(u)=u+\theta(\frac{u}{r})\cdot (\Phi(u)-u),\quad u\in\mathbb{R}^d.\]
It is clear that $\Psi_r$ is well-defined smooth mapping for every sufficiently small $r>0.$ A direct calculation shows that $\mathrm{det}(J_{\Psi_r})\to 1$ in the uniform norm as $r\to0.$ In particular, for sufficiently small $r>0,$ $\mathrm{det}(J_{\Psi_r})$ never vanishes. It follows from Theorem \ref{palais diffeomorphism theorem} that, for sufficiently small $r>0,$ $\Psi_r:\mathbb{R}^d\to\mathbb{R}^d$ is a diffeomorphism. Choose any such $r$ and denote it by $r(0).$ Set $\Phi_0=\Psi_{r(0)}.$ This diffeomorphism obviously satisfies the required properties.	
\end{proof}

\begin{lem}\label{passing from open set to rd} Let $\Omega,\Omega'\subset\mathbb{R}^d$ and let $\Phi:\Omega\to\Omega'$ be a diffeomorphism. Let $B\subset\Omega$ be a ball and let $\Phi_0:\mathbb{R}^d\to\mathbb{R}^d$ be a diffeomorphism such that $\Phi_0=\Phi$ on $B.$ If $T\in B(L_2(\mathbb{R}^d))$ is supported on $B,$ then
\[\mathrm{Ext}_{\Omega'}\Big(U_{\Phi}^{-1}\cdot \mathrm{Rest}_{\Omega}(T)\cdot U_{\Phi}\Big)=U_{\Phi_0}^{-1}TU_{\Phi_0}.\]
\end{lem}
\begin{proof} Indeed, since both sides are continuous in weak operator topology, it suffices to prove the assertion for the case when $T$ is rank $1$ operator.
	
Let 
\[T\xi=\langle \xi,\xi_1\rangle \xi_2,\quad \xi\in L_2(\mathbb{R}^d).\]
where $\xi_1,\xi_2\in L_2(\mathbb{R}^d)$ are supported in $B.$ It is immediate that
\[\Big(U_{\Phi_0}^{-1}TU_{\Phi_0}\Big)\xi=\langle\xi, U_{\Phi_0}^{-1}\xi_1\rangle\cdot U_{\Phi_0^{-1}}\xi_2,\quad \xi\in L_2(\mathbb{R}^d),\]
\[\Big(U_{\Phi}^{-1}\cdot \mathrm{Rest}_{\Omega}(T)\cdot U_{\Phi}\Big)\xi=\langle\xi, U_{\Phi}^{-1}\xi_1\rangle\cdot U_{\Phi^{-1}}\xi_2,\quad \xi\in L_2(U'),\]
\[\Big(\mathrm{Ext}_{\Omega'}\Big(U_{\Phi}^{-1}\cdot \mathrm{Rest}_{\Omega}(T)\cdot U_{\Phi}\Big)\Big)\xi=\langle\xi\cdot\chi_{\Omega'}, U_{\Phi}^{-1}\xi_1\rangle\cdot U_{\Phi^{-1}}\xi_2\cdot\chi_{\Omega'},\quad \xi\in L_2(\mathbb{R}^d).\]
Since $\xi_1$ and $\xi_2$ are supported in $B,$ it follows that expressions in the first and last displays coincide. This proves the assertion for every rank $1$ operator $T$ and, therefore, for every $T.$ 
\end{proof}

\begin{proof}[Proof of Theorem \ref{algebra invariance theorem}] Let the operator $T$ be supported on a compact set $K\subset \Omega.$ Let $t\in K.$ Let diffeomorphism $\Phi_t:\mathbb{R}^d\to\mathbb{R}^d$ and numbers $r_1(t)$ and $r_2(t)$ be as in Lemma \ref{palais lemma}.
	
The collection $\{B(t,r_1(t))\}_{t\in K}$ is an open cover of $K.$ By compactness, one can choose a finite sub-cover. So, let $\{t_n\}_{n=1}^N$ be such that $\{B(t_n,r_1(t_n))\}_{n=1}^N$ be such finite sub-cover. Let $\{\Phi_{t_n}\}_{n=1}^N$ be diffeomorphisms from $\mathbb{R}^d$ to $\mathbb{R}^d$ given by Lemma \ref{palais lemma} so that
\[\Phi(t)=\Phi_{t_n}(t),\quad t\in B(t_n,r_1(t_n)).\]
Let $\{\phi_{n}\}_{n=1}^N$ be such that $\phi_n\in C^{\infty}_c(B(t_n,r_1(t_n)))$ and 
\[\sum_{n=1}^N\phi_n^2=1 \quad \text{on}\ K.\]
Set
\[T_0=\sum_{m=1}^NM_{\phi_m}[M_{\phi_m},T],\quad T_n=M_{\phi_n}TM_{\phi_n},\quad 1\leq n\leq N.\]
We write
\begin{equation}\label{ait eq0}
T=\sum_{n=1}^NM_{\phi_n^2}T=\sum_{n=0}^NT_n.
\end{equation}
If $1\leq n\leq N,$ then $T_n$ is supported on the ball $B(t_n,r_1(t_n)).$ By Lemma \ref{passing from open set to rd}, we have
\[\mathrm{Ext}_{\Omega'}\Big(U_{\Phi}^{-1}\cdot \mathrm{Rest}_{\Omega}(T_n)\cdot U_{\Phi}\Big)=U_{\Phi_{t_n}}^{-1}T_nU_{\Phi_{t_n}}.\]
By Theorem \ref{algebra invariance special}, we have $U_{\Phi_{t_n}}^{-1}T_nU_{\Phi_{t_n}}\in \Pi$ and
\[\mathrm{sym}(U_{\Phi_{t_n}}^{-1}T_nU_{\Phi_{t_n}})=\mathrm{sym}(T_n)\circ\Theta_{\Phi_{t_n}}.\]
Thus,
\begin{equation}\label{ait eq1}
\mathrm{Ext}_{\Omega'}\Big(U_{\Phi}^{-1}\cdot \mathrm{Rest}_{\Omega}(T_n)\cdot U_{\Phi}\Big)\in\Pi
\end{equation}
and
\begin{equation}\label{ait eq2}
\mathrm{sym}\Big(\mathrm{Ext}_{\Omega'}\Big(U_{\Phi}^{-1}\cdot \mathrm{Rest}_{\Omega}(T_n)\cdot U_{\Phi}\Big)\Big)=\mathrm{sym}(T_n)\circ\Theta_{\Phi}.
\end{equation}

If $n=0,$ then $T_0$ is compact by Lemma \ref{commutator is compact}. Clearly, $T_0$ is compactly supported in $\Omega.$ If $A\in B(L_2(\Omega'))$ is compact, then $\mathrm{Ext}_{\Omega'}(A)\in B(L_2(\mathbb{R}^d))$ is also compact (see Notation \ref{s3 first nota} and recall that a composition of bounded and compact operators is compact).  Therefore,
\begin{equation}\label{ait eq3}
\mathrm{Ext}_{\Omega'}\Big(U_{\Phi}^{-1}\cdot \mathrm{Rest}_{\Omega}(T_0)\cdot U_{\Phi}\Big)\in\mathcal{K}(L_2(\mathbb{R}^d))\subset \Pi
\end{equation}
and
\begin{equation}\label{ait eq4}
\mathrm{sym}\Big(\mathrm{Ext}_{\Omega'}\Big(U_{\Phi}^{-1}\cdot \mathrm{Rest}_{\Omega}(T_0)\cdot U_{\Phi}\Big)\Big)=0.
\end{equation}

Combining \eqref{ait eq0}, \eqref{ait eq1} and \eqref{ait eq3}, we obtain
\[\mathrm{Ext}_{\Omega'}\Big(U_{\Phi}^{-1}\cdot \mathrm{Rest}_{\Omega}(T)\cdot U_{\Phi}\Big)=\sum_{n=0}^N\mathrm{Ext}_{\Omega'}\Big(U_{\Phi}^{-1}\cdot \mathrm{Rest}_{\Omega}(T_n)\cdot U_{\Phi}\Big)\in \Pi.\]
Now, combining \eqref{ait eq0}, \eqref{ait eq2} and \eqref{ait eq4}, we obtain
\[\mathrm{sym}\Big(\mathrm{Ext}_{\Omega'}\Big(U_{\Phi}^{-1}\cdot \mathrm{Rest}_{\Omega}(T)\cdot U_{\Phi}\Big)\Big)=\]
\[=\sum_{n=0}^N\mathrm{sym}\Big(\mathrm{Ext}_{\Omega'}\Big(U_{\Phi}^{-1}\cdot \mathrm{Rest}_{\Omega}(T_n)\cdot U_{\Phi}\Big)\Big)=\sum_{n=1}^N\mathrm{sym}(T_n)\circ\Theta_{\Phi}=\]
\[=\mathrm{sym}(T)\circ\Theta_{\Phi}-\mathrm{sym}(T_0)\circ\Theta_{\Phi}=\mathrm{sym}(T)\circ\Theta_{\Phi}.\]
\end{proof}

\section{Principal symbol on compact manifolds}\label{symbol construction section}

\subsection{Globalisation theorem}\label{globalisation subsection}

Globalisation theorem is a folklore. We provide its proof in Appendix \ref{appendix section} for convenience of the reader.

\begin{defi}\label{local algebras def} Let $X$ be a compact manifold with an atlas $\{(\mathcal{U}_i,h_i)\}_{i\in\mathbb{I}}.$ Let $\mathfrak{B}$ be the Borel $\sigma$-algebra on $X$ and let $\nu$ be a countably additive measure on $\mathfrak{B}.$ We say that $\{\mathcal{A}_i\}_{i\in\mathbb{I}}$ are local algebras if
\begin{enumerate}
\item\label{lada} for every $i\in\mathbb{I},$ $\mathcal{A}_i$ is a $\ast$-subalgebra in $B(L_2(X,\nu));$
\item\label{ladb} for every $i\in\mathbb{I},$ elements of $\mathcal{A}_i$ are compactly supported\footnote{This notion is introduced immediately before the Definition \ref{rough principal symbol for compact manifolds}.} in $\mathcal{U}_i;$
\item\label{ladc} for every $i,j\in\mathbb{I},$ if $T\in\mathcal{A}_i$ is compactly supported in $\mathcal{U}_i\cap\mathcal{U}_j,$ then $T\in\mathcal{A}_j;$
\item\label{ladd} for every $i\in\mathbb{I},$ if $T\in \mathcal{K}(L_2(X,\nu))$ is compactly supported in $\mathcal{U}_i,$ then $T\in\mathcal{A}_i;$
\item\label{lade} for every $i\in\mathbb{I},$ if $\phi\in C_c(\mathcal{U}_i),$ then $M_{\phi}\in\mathcal{A}_i;$
\item\label{ladf} for every $i\in\mathbb{I},$ if $\phi\in C_c(\mathcal{U}_i),$ then the closure of $M_{\phi}\mathcal{A}_iM_{\phi}$ in the uniform norm is contained in  $\mathcal{A}_i;$
\item\label{ladg} for every $i\in\mathbb{I},$ if $T\in\mathcal{A}_i$ and if $\phi\in C_c(\mathcal{U}_i),$ then $[T,M_{\phi}]\in \mathcal{K}(L_2(X,\nu)).$
\end{enumerate}
\end{defi}

\begin{defi}\label{global algebra def} In the setting of Definition \ref{local algebras def}, we say that $T\in\mathcal{A}$ if
\begin{enumerate}
\item\label{gada} for every $i\in\mathbb{I}$ and for every $\phi\in C_c(\mathcal{U}_i),$ we have $M_{\phi}TM_{\phi}\in\mathcal{A}_i;$
\item\label{gadb} for every $\psi\in C(X),$ the commutator $[T,M_{\psi}]$ is compact.
\end{enumerate}
\end{defi}

\begin{defi}\label{local homomorphism def} Let $\mathcal{B}$ be a $\ast$-algebra. In the setting of Definition \ref{local algebras def}, $\{\mathrm{hom}_i\}_{i\in\mathbb{I}}$ are called local homomorphisms if 
\begin{enumerate}
\item\label{lhda} for every $i\in\mathbb{I},$ $\mathrm{hom}_i:\mathcal{A}_i\to  \mathcal{B}$ is a $\ast$-homomorphism;
\item\label{lhdb} for every $i,j\in\mathbb{I},$ we have
$\mathrm{hom}_i=\mathrm{hom}_j$ on $\mathcal{A}_i\cap\mathcal{A}_j;$
\item\label{lhdc} $T\in\mathcal{A}_i$ is compact iff $\mathrm{hom}_i(T)=0;$
\item\label{lhdd} there exists a $\ast$-homomorphism $\mathrm{Hom}:C(X)\to\mathcal{B}$ such that 
\[\mathrm{hom}_i(M_{\phi})=\mathrm{Hom}(\phi),\quad \phi\in C_c(\mathcal{U}_i),\quad i\in\mathbb{I}.\]
\end{enumerate}
\end{defi}

\begin{thm}\label{globalisation theorem} In the setting of Definitions \ref{local algebras def}, \ref{local homomorphism def} and \ref{global algebra def}, we have
\begin{enumerate}
\item\label{gta} $\mathcal{A}$ is a unital $C^{\ast}$-subalgebra in $B(L_2(X,\nu))$ which contains $\mathcal{A}_i$ for every $i\in\mathbb{I}$ and $\mathcal{K}(L_2(X,\nu));$
\item\label{gtb} there exists a $\ast$-homomorphism $\mathrm{hom}:\mathcal{A}\to \mathcal{B}$ such that
\begin{enumerate}
\item $\mathrm{hom}=\mathrm{hom}_i$ on $\mathcal{A}_i$ for every $i\in\mathbb{I};$
\item $\mathrm{ker}(\mathrm{hom})=\mathcal{K}(L_2(X,\nu));$
\end{enumerate}
\item\label{gtc} $\ast$-homomorphism as in \eqref{gtb} is unique.
\end{enumerate}
\end{thm}

\subsection{Construction of the principal symbol mapping}\label{construction subsection}

Let $\mathfrak{B}$ be the Borel $\sigma$-algebra on the manifold $X$ and let $\nu:\mathfrak{B}\to\mathbb{R}$ be a continuous positive density. 

It is immediate that the mapping $h_i : (\mathcal{U}_i,\nu)\to(\Omega_i,\nu\circ h_i^{-1})$ preserves the measure. Define an isometry $W_i:L_2(\mathcal{U}_i,\nu)\to L_2(\Omega_i,\nu\circ h_i^{-1})$ by setting
\[W_if=f\circ h_i^{-1},\quad f\in L_2(\mathcal{U}_i,\nu).\]
If $T$ is compactly supported in $\mathcal{U}_i,$ then $W_iTW_i^{-1}$ is understood as an element of the algebra $B(L_2(\Omega_i,\nu\circ h_i^{-1})).$ The latter operator is compactly supported in $\Omega_i.$ By Definition \ref{condition on nu}, exactly the same operator also belongs to $B(L_2(\Omega_i))$ and, therefore, can be extended to an element $\mathrm{Ext}_{\Omega_i}(W_iTW_i^{-1})$ of $B(L_2(\mathbb{R}^d)).$ For the notion $\mathrm{Ext}_{\Omega_i}$ we refer to Notation \ref{s3 second nota}.

\begin{defi}\label{localised principal-symbol-able operators} Let $X$ be a smooth compact manifold and let $\nu$ be a continuous positive density on $X.$ For every $i\in\mathbb{I},$ let $\Pi_i$ consist of the operators $T\in B(L_2(X,\nu))$ compactly supported in $\mathcal{U}_i$ and such that
\[\mathrm{Ext}_{\Omega_i}(W_iTW_i^{-1})\in \Pi.\]
\end{defi}

For example, every operator $M_{\phi},$ $\phi\in C_c(\mathcal{U}_i)$ belongs to $\Pi_i.$

For notation $C(S^{\ast}X)$ below we refer to Definition \ref{cosphere bundle def}. For the notion $\mathrm{sym},$ we refer to \eqref{euclidean symbol definition}.

\begin{defi}\label{principal symbol for compact manifold} Let $X$ be a smooth compact manifold and let $\nu$ be a continuous positive density on $X.$ For every $i\in\mathbb{I},$ the mapping $\mathrm{sym}_i:\Pi_i\to C(S^{\ast}X)$ is defined by the formula
\[\mathrm{sym}_i(T)=\mathrm{sym}(\mathrm{Ext}_{\Omega_i}(W_iTW_i^{-1}))\circ H_i,\quad T\in\Pi_i.\]
\end{defi}

\begin{thm}\label{existence of principal symbol thm} Let $X$ be a smooth compact manifold and let $\nu$ be a continuous positive density on $X.$
\begin{enumerate}
\item\label{epsta} Collection $\{\Pi_i\}_{i\in\mathbb{I}}$ introduced in Definition \ref{localised principal-symbol-able operators} satisfies all the conditions in Definition \ref{local algebras def};
\item\label{epstb} Collection $\{\mathrm{sym}_i\}_{i\in\mathbb{I}}$ introduced in Definition \ref{principal symbol for compact manifold} satisfies all the conditions in Definition \ref{local homomorphism def}.
\end{enumerate}
\end{thm}

That is, the collection $\{\Pi_i\}_{i\in\mathbb{I}}$ of $\ast$-algebras and the collection $\{\mathrm{sym}_i\}_{i\in\mathbb{I}}$ of $\ast$-homo\-morphisms satisfy the  conditions in Theorem \ref{globalisation theorem}. 

Definition \ref{principal symbol definition thm} below is the culmination of the paper. Having this definition at hands, we easily prove Theorem \ref{symbol manifold intro thm}.

\begin{defi}\label{principal symbol definition thm} Let $X$ be a smooth compact Riemannian manifold and let $\nu$ be a continuous positive density on $X.$
\begin{enumerate}
\item The domain $\Pi_X$ of the principal symbol mapping is the $C^{\ast}$-algebra constructed in Theorem \ref{globalisation theorem} from the collection $\{\Pi_i\}_{i\in\mathbb{I}}.$
\item The principal symbol mapping $\mathrm{sym}_X:\Pi_X\to C(S^{\ast}X)$ is the $\ast$-homomorphism constructed in Theorem \ref{globalisation theorem} from the collection $\{\mathrm{sym}_i\}_{i\in\mathbb{I}}.$
\end{enumerate}
\end{defi}

\subsection{Proof of Theorem \ref{existence of principal symbol thm}}

Lemma \ref{pii is staralgebra} below delivers verification of condition \eqref{lada} in Definitions \ref{local algebras def} and \ref{local homomorphism def}.

\begin{lem}\label{pii is staralgebra} For every $i\in\mathbb{I},$ we have
\begin{enumerate}
\item $\Pi_i$ is a $\ast$-subalgebra in $B(L_2(X,\nu));$
\item $\mathrm{sym}_i:\Pi_i\to C(S^{\ast}X)$ is a $\ast$-homomorphism.
\end{enumerate}	
\end{lem}
\begin{proof} It is immediate that $\Pi_i$ is a subalgebra in $B(L_2(X,\nu))$ and that $\mathrm{sym}_i:\Pi_i\to C(S^{\ast}X)$ is a homomorphism. We need to show that $\Pi_i$ is closed with respect to taking adjoints and that $\mathrm{sym}_i$ is invariant with respect to this operation.

Let $T\in\Pi_i$ and let us show that $T^{\ast}\in\Pi_i.$ Recall that, due to the Condition \ref{condition on nu}, $\nu\circ h_i^{-1}$ is absolutely continuous and that its density denoted by $a_i$ as well as its inverse  $a_i^{-1}$ are assumed to be continuous in $\Omega_i.$ The following equality\footnote{The operators $M_{a_i}$ and $M_{a_i^{-1}}$ are unbounded. The equality should be understood as
$LHS\xi=RHS\xi$ for every compactly supported $\xi\in L_2(\Omega_i).$

Indeed, for such $\xi,$ we have $\xi_1=M_{a_i}\xi\in L_2(\Omega_i).$ Since $T$ is compactly supported in $\mathcal{U}_i,$ it follows that $(W_iTW_i^{-1})^{\ast}$ is compactly supported in $\Omega_i.$ Hence, the function $\xi_2=(W_iTW_i^{-1})^{\ast}\xi_1$ is compactly supported in $\Omega_i.$ Hence, the function $M_{a_i^{-1}}\xi_2$ belongs to $L_2(\Omega_i)$ and the right hand side of \eqref{wi with adjoints} makes sense.} is easy to verify directly.
\begin{equation}\label{wi with adjoints}
W_iT^{\ast}W_i^{-1}=M_{a_i^{-1}}\cdot (W_iTW_i^{-1})^{\ast}\cdot M_{a_i}.
\end{equation}
However, by Definition \ref{localised principal-symbol-able operators}, the operator $T$ is compactly supported in $\mathcal{U}_i.$ Hence, the operator $(W_iTW_i^{-1})^{\ast}$ is compactly supported in $\Omega_i.$ Choose $\phi\in C_c(\Omega_i)$ such that 
\[(W_iTW_i^{-1})^{\ast}=M_{\phi}\cdot (W_iTW_i^{-1})^{\ast}=(W_iTW_i^{-1})^{\ast}\cdot M_{\phi}.\]
Thus,
\[W_iT^{\ast}W_i^{-1}=M_{a_i^{-1}\phi}\cdot (W_iTW_i^{-1})^{\ast}\cdot M_{a_i\phi}.\]
Thus,
\begin{equation}\label{pii eq0}
\mathrm{Ext}_{\Omega_i}(W_iT^{\ast}W_i^{-1})=M_{a_i^{-1}\phi}\cdot (\mathrm{Ext}_{\Omega_i}(W_iTW_i^{-1}))^{\ast}\cdot M_{a_i\phi}.
\end{equation}
Since $a_i\phi,a_i^{-1}\phi\in C_c(\mathbb{R}^d),$ it follows that every factor in the right hand side of \eqref{pii eq0} belongs to $\Pi.$ Hence, so is the expression on the left hand side. In other words, $T^{\ast}\in\Pi_i.$ Thus, $\Pi_i$ is closed with respect to taking adjoints.

Recall that (by \cite{DAO1}) $\mathrm{sym}$ is a $\ast$-homomorphism. Applying $\mathrm{sym}$ to the equality \eqref{pii eq0}, we obtain
\[\mathrm{sym}_i(T^{\ast})=\mathrm{sym}(\mathrm{Ext}_{\Omega_i}(W_iT^{\ast}W_i^{-1}))=\]
\[=\mathrm{sym}(M_{a_i^{-1}\phi})\cdot \mathrm{sym}((\mathrm{Ext}_{\Omega_i}(W_iTW_i^{-1}))^{\ast})\cdot \mathrm{sym}(M_{a_i\phi})=\]
\[=\mathrm{sym}(M_{\phi^2})\cdot \mathrm{sym}(\mathrm{Ext}_{\Omega_i}(W_iTW_i^{-1}))^{\ast}=\]
\[=\mathrm{sym}(\mathrm{Ext}_{\Omega_i}(M_{\phi^2}\cdot W_iTW_i^{-1}))^{\ast}.\]
It is clear that
\[M_{\phi^2}\cdot W_iTW_i^{-1}=W_iTW_i^{-1}.\]
Thus,
\[\mathrm{sym}_i(T^{\ast})=\mathrm{sym}(\mathrm{Ext}_{\Omega_i}(W_iTW_i^{-1}))^{\ast}=\mathrm{sym}_i(T)^{\ast}.\]
\end{proof}

In the following lemma, $\Xi_{\Phi_{i,j}}$ is defined according to the Notation \ref{s3 notation}.

\begin{lem}\label{xi vs h lemma} For every $i,j\in\mathbb{I},$ we have
\[\Xi_{\Phi_{i,j}}=H_i\circ H_j^{-1}\]
on $\Omega_{j,i}\times\mathbb{R}^d.$	
\end{lem}
\begin{proof} By Notation \ref{s3 notation},
\[\Xi_{\Phi_{i,j}}(t,s)=(\Phi_{i,j}^{-1}(t),J_{\Phi_{i,j}^{\ast}(\Phi_{i,j}^{-1}(t))}s).\]
By the chain rule, we have
\[J_{\Phi_{i,j}}(\Phi_{i,j}^{-1}(t))\cdot J_{\Phi_{i,j}^{-1}}(t)=J_{\Phi_{i,j}\circ\Phi_{i,j}^{-1}}(t)=1_{M_d(\mathbb{C})}.\]
Taking into account that $\Phi_{i,j}^{-1}=\Phi_{j,i},$ we write 
\[J_{\Phi_{i,j}(\Phi_{i,j}^{-1}(t))}=(J_{\Phi_{j,i}})^{-1}(t)\mbox{ and }J_{\Phi_{i,j}(\Phi_{i,j}^{-1}(t))}^{\ast}=(J_{\Phi_{j,i}}^{\ast})^{-1}(t).\]
\end{proof}

The following lemma verifies the condition \eqref{ladc} in Definition \ref{local algebras def} and condition \eqref{lhdb} in Definition \ref{local homomorphism def}.

\begin{lem}\label{symbol on manifold transformation lemma} Let $(\mathcal{U}_i,h_i)$ and $(\mathcal{U}_j,h_j)$ be charts. Let $T\in B(L_2(X,\nu))$ be compactly supported in $\mathcal{U}_i\cap\mathcal{U}_j.$ 
\begin{enumerate}
\item If $T\in\Pi_i,$ then $T\in\Pi_j;$
\item We have $\mathrm{sym}_i(T)=\mathrm{sym}_j(T).$
\end{enumerate}
\end{lem}
\begin{proof} Let $V_{\Phi}\xi=\xi\circ\Phi$ (provided that the image of the mapping $\Phi$ is contained in the domain of the function $\xi$). Since $W_j=V_{\Phi_{i,j}}^{-1}W_i,$ it follows that (using the Notation \ref{s3 notation})
\[W_jTW_j^{-1}=V_{\Phi_{i,j}}^{-1}\cdot W_iTW_i^{-1}\cdot V_{\Phi_{i,j}}=\]
\[=U_{\Phi_{i,j}}^{-1}\cdot M_{|J_{\Phi_{i,j}}|^{\frac12}}\cdot W_iTW_i^{-1} \cdot M_{|J_{\Phi_{i,j}}|^{-\frac12}}\cdot U_{\Phi_{i,j}}.\]
	
Since $T$ is compactly supported in $\mathcal{U}_i\cap\mathcal{U}_j,$ it follows that $W_iTW_i^{-1}\in \mathrm{Rest}_{\Omega_i}(\Pi)$ is compactly supported in $\Omega_{i,j}.$ Let $A\subset\Omega_{i,j}$ be compact and such that
\[M_{\chi_A}\cdot W_iTW_i^{-1}\cdot M_{\chi_A}=W_iTW_i^{-1}.\]
Using Tietze extension theorem, choose $\phi\in C(\mathbb{R}^d)$ such that $\phi^{-1}\in C(\mathbb{R}^d)$ and such that $\phi=|J_{\Phi_{i,j}}|^{\frac12}$ on $A.$ It follows that
\[M_{|J_{\Phi_{i,j}}|^{\frac12}}\cdot W_iTW_i^{-1} \cdot M_{|J_{\Phi_{i,j}}|^{-\frac12}}=M_{\phi}\cdot W_iTW_i^{-1} \cdot M_{\phi^{-1}}=\]
\[=W_iTW_i^{-1}+[M_{\phi},W_iTW_i^{-1}] \cdot M_{\phi^{-1}}\stackrel{L.\ref{commutator is compact}}{\in} W_iTW_i^{-1}+\mathcal{K}(L_2(\Omega_i)).\]

Combining the preceding paragraphs, we conclude that 
\[W_jTW_j^{-1}\in U_{\Phi_{i,j}}^{-1}\cdot W_iTW_i^{-1}\cdot U_{\Phi_{i,j}}+\mathcal{K}(L_2(\Omega_j)).\]
Denote for brevity 
\[T_i=\mathrm{Ext}_{\Omega_i}(W_iTW_i^{-1}),\quad T_j=\mathrm{Ext}_{\Omega_j}(W_jTW_j^{-1}).\]
The preceding display can be now re-written as
\[\mathrm{Rest}_{\Omega_j}(T_j)\in U_{\Phi_{i,j}}^{-1}\cdot \mathrm{Rest}_{\Omega_i}(T_i)\cdot U_{\Phi_{i,j}}+\mathcal{K}(L_2(\Omega_j)).\]
Thus,
\[T_j\in \mathrm{Ext}_{\Omega_j}\Big(U_{\Phi_{i,j}}^{-1}\cdot \mathrm{Rest}_{\Omega_i}(T_i)\cdot U_{\Phi_{i,j}}\Big)+\mathcal{K}(L_2(\mathbb{R}^d)).\]
By Theorem \ref{algebra invariance theorem}, we have
\[\mathrm{Ext}_{\Omega_j}\Big(U_{\Phi_{i,j}}^{-1}\cdot \mathrm{Rest}_{\Omega_i}(T_i)\cdot U_{\Phi_{i,j}}\Big)\in \Pi\]
and
\[\mathrm{sym}\Big(\mathrm{Ext}_{\Omega_j}\Big(U_{\Phi_{i,j}}^{-1}\cdot \mathrm{Rest}_{\Omega_i}(T_i)\cdot U_{\Phi_{i,j}}\Big)\Big)=\mathrm{sym}(T_i)\circ\Xi_{\Phi_{i,j}}.\]
By Lemma \ref{compact operators lemma}, compact operators belong to $\Pi.$ Therefore, $T_j\in \Pi$ and
\[\mathrm{sym}(T_j)=\mathrm{sym}(T_i)\circ\Xi_{\Phi_{i,j}}\stackrel{L.\ref{xi vs h lemma}}{=}\mathrm{sym}(T_i)\circ H_i\circ H_j^{-1}.\]
Finally,
\[\mathrm{sym}_j(T)=\mathrm{sym}(T_j)\circ H_j=\mathrm{sym}(T_i)\circ H_i\circ H_j^{-1}\circ H_j=\mathrm{sym}(T_i)\circ H_i=\mathrm{sym}_i(T).\]
\end{proof}

\begin{proof}[Proof of Theorem \ref{existence of principal symbol thm}  \eqref{epsta}] The condition \eqref{lada} in Definition \ref{local algebras def} is verified in Lemma \ref{pii is staralgebra}. The condition \eqref{ladb} in Definition \ref{local algebras def} is immediate. The condition \eqref{ladc} in Definition \ref{local algebras def} is verified in Lemma \ref{symbol on manifold transformation lemma}.

Let us verify the condition \eqref{ladd} in Definition \ref{local algebras def}. If $i\in\mathbb{I}$ and if $T\in\mathcal{K}(L_2(X,\nu))$ is compactly supported in $\mathcal{U}_i,$ then $\mathrm{Ext}_{\Omega_i}(W_iTW_i^{-1})\in\mathcal{K}(L_2(\mathbb{R}^d)).$ Using Lemma \ref{compact operators lemma}, we conclude that $\mathrm{Ext}_{\Omega_i}(W_iTW_i^{-1})\in\Pi.$ In other words, $T\in\Pi_i.$

The condition \eqref{lade} in Definition \ref{local algebras def} is immediate.

Let us verify the condition \eqref{ladf} in Definition \ref{local algebras def}. Let $i\in\mathbb{I}$ and let $\phi\in C_c(\mathcal{U}_i).$ Suppose $\{T_n\}_{n\geq1}\subset M_{\phi}\Pi_iM_{\phi}$ are such that $T_n\to T$ in the uniform norm. It follows that $T$ is compactly supported in $\mathcal{U}_i$ and
\[\mathrm{Ext}_{\Omega_i}(W_iT_nW_i^{-1})\to \mathrm{Ext}_{\Omega_i}(W_iTW_i^{-1}),\quad n\to\infty,\]
in the uniform norm. The sequence on the left hand side is in $\Pi.$ Hence, so is its limit. In other words, $T\in\Pi_i.$

Let us verify the condition \eqref{ladg} in Definition \ref{local algebras def}. Let $i\in\mathbb{I}$ and let $T\in\Pi_i$ and $\phi\in C_c(\mathcal{U}_i).$ Let $\psi=\phi\circ h_i^{-1}\in C_c(\mathbb{R}^d).$ We have
\[\mathrm{Ext}_{\Omega_i}(W_i[T,M_{\phi}]W_i^{-1})=[\mathrm{Ext}_{\Omega_i}(W_iTW_i^{-1}),M_{\psi}].\]
Since $\mathrm{Ext}_{\Omega_i}(W_iTW_i^{-1})\in\Pi,$ it follows that the commutator on the right hand side is compact by Lemma \ref{commutator is compact}. Therefore, the operator on the left hand side is compact and, therefore, so is $[T,M_{\phi}].$
\end{proof}

\begin{proof}[Proof of Theorem \ref{existence of principal symbol thm} \eqref{epstb}] The condition \eqref{lhda} in Definition \ref{local homomorphism def} is verified in Lemma \ref{pii is staralgebra}. The condition \eqref{lhdb} in Definition \ref{local homomorphism def} is verified in Lemma \ref{symbol on manifold transformation lemma}.

Let us verify the condition \eqref{lhdc} in Definition \ref{local homomorphism def}. If $T\in\Pi_i$ is compact, then so is $\mathrm{Ext}_{\Omega_i}(W_iTW_i^{-1}).$ Since $\mathrm{sym}$ vanishes on compact operators, it follows that
\[\mathrm{sym}_i(T)\stackrel{D.\ref{principal symbol for compact manifold}}{=}\mathrm{sym}(\mathrm{Ext}_{\Omega_i}(W_iTW_i^{-1}))\circ H_i=0\circ H_i=0.\]
Conversely, if $T\in\Pi_i$ is such that $\mathrm{sym}_i(T)=0,$ then
\[\mathrm{sym}(\mathrm{Ext}_{\Omega_i}(W_iTW_i^{-1}))=0.\]
Since $\mathrm{ker}(\mathrm{sym})=\mathcal{K}(L_2(\mathbb{R}^d)),$ it follows that 
\[\mathrm{Ext}_{\Omega_i}(W_iTW_i^{-1})\in \mathcal{K}(L_2(\mathbb{R}^d)).\]
Thus, $T\in\mathcal{K}(L_2(X,\nu)).$

The condition \eqref{lhdd} in Definition \ref{local homomorphism def} is immediate if we take $\mathrm{Hom}$ to be the natural embedding $C(X)\to C(S^{\ast}X).$
\end{proof}

\subsection{Proof of Theorem \ref{symbol manifold intro thm}}

\begin{proof}[Proof of Theorem \ref{symbol manifold intro thm}] By Definition \ref{principal symbol definition thm}, $\Pi_X$ is a $C^{\ast}$-algebra and the mapping $\mathrm{sym}_X:\Pi_X\to C(S^{\ast}X)$ is a $\ast$-homomorphism. By Definition \ref{principal symbol definition thm} and Theorem \ref{globalisation theorem} \eqref{gtb}, $\mathrm{ker}(\mathrm{sym}_X)=\mathcal{K}(L_2(X,\nu)).$
	
Let us show that $\mathrm{sym}_X$ is surjective. Denote the image of $\mathrm{sym}_X$ by $A$ and note that $A$ is a $C^{\ast}$-subalgebra in $C(S^{\ast}X).$ Let $F\in C^{\infty}(S^{\ast}X).$ Let $(\phi_n)_{n=1}^N$ be a good\footnote{See Definition \ref{good partition of unity}.} partition of unity so that $\phi_n\in C^{\infty}_c(\mathcal{U}_{i_n})$ for $1\leq n\leq N.$ It follows that
\[q_n=(F\phi_n)\circ H_{i_n}^{-1}\in C_c(\Omega_{i_n}\times\mathbb{R}^d).\]
By Lemma \ref{81 lemma}, we have $T_{q_n}\in \Pi$ and $\mathrm{sym}(T_{q_n})=q_n.$ Let $\psi_n\in C_c(\Omega_{i_n})$ be such that $\phi_n=\phi_n\psi_n.$ We have $T_n=M_{\psi_n}T_{q_n}M_{\psi_n}\in\Pi$ and $\mathrm{sym}(T_n)=q_n\psi_n^2=q_n.$ Since $T_n$ is (bounded and) compactly supported in $\Omega_{i_n},$ it follows that $S_n=W_{i_n}^{-1}\mathrm{Rest}_{\Omega_i}(T_n)W_{i_n}$ is bounded and compactly supported in $\mathcal{U}_{i_n}.$ It is clear that $S_n\in \Pi_{i_n}\subset \Pi_X$ and that
\[\mathrm{sym}_X(S_n)=\mathrm{sym}(\mathrm{Ext}_{\Omega_{i_n}}(W_{i_n}T_nW_{i_n}^{-1}))\circ H_{i_n}=\mathrm{sym}(T_n)\circ H_{i_n}=q_n\circ H_{i_n}=F\phi_n.\]
Thus, $S=\sum_{n=1}^NS_n\in\Pi_X$ and
\[\mathrm{sym}_X(S)=\sum_{n=1}^N\mathrm{sym}_X(S_n)=\sum_{n=1}^NF\phi_n=F.\]
Hence, every $F\in C^{\infty}(S^{\ast}X)$ belongs to the $A.$ In other words, $C^{\infty}(S^{\ast}X)\subset A.$ Since $A$ is a $C^{\ast}$-subalgebra in $C(S^{\ast}X),$ it follows that $A=C(S^{\ast}X).$ Hence, $\mathrm{sym}_X$ is surjective.
\end{proof}

\section{Proof of the Connes Trace Theorem}\label{ctt section}

\begin{lem}\label{do power lemma} Let $g$ be as in Theorem \ref{lb rd self-adjoint}. Let $\phi\in C^{\infty}_c(\mathbb{R}^d).$ We have
\[M_{\phi}(1-\Delta_g)^{-\frac{r}{2}}\in\mathcal{L}_{\frac{d}{r},\infty}.\]
\end{lem}
\begin{proof} By definition, the principal symbol of $1-\Delta_g$ is 
\[p_0:(t,s)\to \langle g(t)^{-1}s,s\rangle,\quad t,s\in\mathbb{R}^d.\]
By Theorem \ref{complex power thm}, we have
\[(1-\Delta_g)^{-\frac{r}{2}}\in\Psi^{-r}(\mathbb{R}^d).\]
We now write
\[M_{\phi}(1-\Delta_g)^{-\frac{r}{2}}=M_{\phi}(1-\Delta)^{-\frac{r}{2}}\cdot (1-\Delta)^{\frac{r}{2}}(1-\Delta_g)^{-\frac{r}{2}}.\]
When $r>\frac{d}{2},$ the first factor belongs to $\mathcal{L}_{\frac{d}{r},\infty}$ by Theorem 1.4 in \cite{LeSZ}. When $r=\frac{d}{2},$ the first factor belongs to $\mathcal{L}_{\frac{d}{r},\infty}$ by Theorem 1.3 in \cite{LeSZ}. When $r<\frac{d}{2},$ the first factor belongs to $\mathcal{L}_{\frac{d}{r},\infty}$ by Theorem 1.1 in \cite{LeSZ} (applied with $r<\frac{d}{2}$). The second factor belongs to $\Psi^0(\mathbb{R}^d)$ and is, therefore, bounded.
\end{proof}

\begin{lem}\label{do power psymbol lemma} Let $g$ be as in Theorem \ref{lb rd self-adjoint}. We have
\[M_{\psi}(1-\Delta_g)^{-\frac{d}{2}}(1-\Delta)^{\frac{d}{2}}\in\Pi,\]
\[\mathrm{sym}\big(M_{\psi}(1-\Delta_g)^{-\frac{d}{2}}(1-\Delta)^{\frac{d}{2}}\big)(t,s)=\psi(t)\langle g(t)^{-1}s,s\rangle^{-\frac{d}{2}}.\]
\end{lem}
\begin{proof} By definition, the principal symbol of $1-\Delta_g$ is 
\[p_0:(t,s)\to \langle g(t)^{-1}s,s\rangle,\quad t,s\in\mathbb{R}^d.\]
By Theorem \ref{complex power thm}, we have
\[(1-\Delta_g)^{-\frac{d}{2}}=\mathrm{Op}((p_0+1)^{-\frac{d}{2}})+\mathrm{Err}_0,\quad \mathrm{Err}_0\in\Psi^{-d-1}(\mathbb{R}^d).\]
Let
\[p_1:(t,s)\to \psi(t)(1+\langle g(t)^{-1}s,s\rangle)^{-\frac{d}{2}}(1+|s|^2)^{-\frac{d}{2}},\quad t,s\in\mathbb{R}^d.\]
Clearly,
\[M_{\psi}\cdot \mathrm{Op}((p_0+1)^{-\frac{d}{2}})\cdot (1-\Delta)^{\frac{d}{2}}=\mathrm{Op}(p_1).\]
Thus,
\[M_{\psi}(1-\Delta_g)^{-\frac{d}{2}}(1-\Delta)^{\frac{d}{2}}=\mathrm{Op}(p_1)+\mathrm{Err}_1,\quad \mathrm{Err}_1\in\Psi^{-1}(\mathbb{R}^d).\]
Since both 
\[M_{\psi}(1-\Delta_g)^{-\frac{d}{2}}(1-\Delta)^{\frac{d}{2}}\mbox{ and }\mathrm{Op}(p_1)\]
are compactly supported from the left, it follows that so is $\mathrm{Err}_1.$ Thus, $\mathrm{Err}_1$ is a compact operator. Consequently, $\mathrm{Err}_1\in\Pi$ and $\mathrm{sym}(\mathrm{Err}_1)=0.$

Let
\[p_2(t,s)=\psi(t)\langle g(t)^{-1}s,s\rangle^{-\frac{d}{2}},\quad t\in\mathbb{R}^d,\quad s\in\mathbb{S}^{d-1}.\]
By Lemma \ref{82 lemma}, we have that $\mathrm{Op}(p_1)-T_{p_2}$ is compact. So, our operator belongs to $\Pi$ and its symbol equals that of $T_{p_2},$ i.e. equals $p_2.$
\end{proof}

\begin{lem}\label{ctt simplification lemma} Let $(X,G)$ be a compact Riemannian manifold. Let $\psi\in C^{\infty}(X)$ be compactly supported in the chart $(\mathcal{U}_i,h_i).$ Let $\hat{g}_i:\mathbb{R}^d\to \mathrm{GL}^+(d,\mathbb{R})$ be as in Theorem \ref{lb rd self-adjoint} and such that $\hat{g}_i=g_i$ in the neighborhood of the support of $\psi\circ h_i^{-1}.$  We have
\[\mathrm{Ext}_{\Omega_i}(W_iM_{\psi}^2(1-\Delta_G)^{-1}M_{\psi}^2W_i^{-1})-M_{\psi\circ h_i^{-1}}^2(1-\Delta_{\hat{g}_i})^{-1}M_{\psi\circ h_i^{-1}}^2\in\mathcal{L}_{\frac{d}{3},\infty}.\]
\end{lem}
\begin{proof} Let $\Omega_i'\subset\Omega_i$ be a compact set such that $\psi\circ h_i^{-1}$ is supported in $\Omega_i'$ and such that $g_i=\hat{g}_i$ on $\Omega_i'.$ Let $\phi\in C_c(\mathbb{R}^d)$ be supported in $\Omega_i'$ and such that $\phi\cdot (\psi\circ h_i^{-1})=\psi\circ h_i^{-1}.$ We write
\[\mathrm{Ext}_{\Omega_i}(W_iM_{\psi}^2(1-\Delta_G)^{-1}M_{\psi}^2W_i^{-1})=M_{\phi}^2\cdot\mathrm{Ext}_{\Omega_i}(W_iM_{\psi}^2(1-\Delta_G)^{-1}M_{\psi}^2W_i^{-1})=\]
\[=M_{\phi}(1-\Delta_{\hat{g}_i})^{-1}\cdot (1-\Delta_{\hat{g}_i})M_{\phi}\mathrm{Ext}_{\Omega_i}(W_iM_{\psi}^2(1-\Delta_G)^{-1}M_{\psi}^2W_i^{-1}).\]
It follows directly from Definition \ref{lb manifold def} that
\[(1-\Delta_{\hat{g}_i})M_{\phi}=\mathrm{Ext}_{\Omega_i}(W_i(1-\Delta_G)M_{\phi\circ h_i}W_i^{-1}).\]
Thus,
\[(1-\Delta_{\hat{g}_i})M_{\phi}\mathrm{Ext}_{\Omega_i}(W_iM_{\psi}^2(1-\Delta_G)^{-1}M_{\psi}^2W_i^{-1})=\]
\[=\mathrm{Ext}_{\Omega_i}(W_i(1-\Delta_g)M_{\phi\circ h_i}W_i^{-1})\cdot\mathrm{Ext}_{\Omega_i}(W_iM_{\psi}^2(1-\Delta_G)^{-1}M_{\psi}^2W_i^{-1})=\]
\[=\mathrm{Ext}_{\Omega_i}(W_i(1-\Delta_G)M_{\psi}^2(1-\Delta_G)^{-1}M_{\psi}^2W_i^{-1}).\]
Combining these equalities, we obtain
\[\mathrm{Ext}_{\Omega_i}(W_iM_{\psi}^2(1-\Delta_g)^{-1}M_{\psi}^2W_i^{-1})=\]
\[=M_{\phi}(1-\Delta_{\hat{g}_i})^{-1}\cdot\mathrm{Ext}_{\Omega_i}(W_i(1-\Delta_G)M_{\psi}^2(1-\Delta_G)^{-1}M_{\psi}^2W_i^{-1}).\]

Now,
\[(1-\Delta_G)M_{\psi}^2(1-\Delta_G)^{-1}M_{\psi}^2=M_{\psi}^4-[\Delta_G,M_{\psi}^2](1-\Delta_G)^{-1}M_{\psi}^2.\]
Thus,
\[\mathrm{Ext}_{\Omega_i}(W_iM_{\psi}^2(1-\Delta_G)^{-1}M_{\psi}^2W_i^{-1})=M_{\phi}(1-\Delta_{\hat{g}_i})^{-1}M_{\psi\circ h_i^{-1}}^4-\]
\[-M_{\phi}(1-\Delta_{\hat{g}_i})^{-1}\cdot \mathrm{Ext}_{\Omega_i}(W_i[\Delta_G,M_{\psi}^2](1-\Delta_G)^{-1}M_{\psi}^2W_i^{-1}).\]

Since $X$ is compact, it follows that
\[(1-\Delta_G)^{-1}:L_2(X)\to W^{2,2}(X),\quad [\Delta_G,M_{\psi}^2]:W^{2,2}(X)\to W^{1,2}(X)\]
are bounded operators. We now write
\[[\Delta_G,M_{\psi}^2](1-\Delta_G)^{-1}=(1-\Delta_G)^{-\frac12}\cdot (1-\Delta_G)^{\frac12}[\Delta_G,M_{\psi}^2](1-\Delta_G)^{-1},\]
where the first factor is in $\mathcal{L}_{d,\infty}$ and the second factor is bounded. By Lemma \ref{do power lemma}, we have
\[M_{\phi}(1-\Delta_{\hat{g}_i})^{-1}\in\mathcal{L}_{\frac{d}{2},\infty}.\]
By H\"older inequality, we have
\[M_{\phi}(1-\Delta_{\hat{g}_i})^{-1}\cdot \mathrm{Ext}_{\Omega_i}(W_i[\Delta_g,M_{\psi}^2](1-\Delta_G)^{-1}M_{\psi}^2W_i^{-1})\in\mathcal{L}_{\frac{d}{3},\infty}.\]
Thus,
\[\mathrm{Ext}_{\Omega_i}(W_iM_{\psi}^2(1-\Delta_G)^{-1}M_{\psi}^2W_i^{-1})-M_{\phi}(1-\Delta_{\hat{g}_i})^{-1}M_{\psi\circ h_i^{-1}}^4\in\mathcal{L}_{\frac{d}{3},\infty}.\]
Note that
\[M_{\phi}(1-\Delta_{\hat{g}_i})^{-1}M_{\psi\circ h_i^{-1}}^4=M_{\psi\circ h_i^{-1}}^2(1-\Delta_{\hat{g}_i})^{-1}M_{\psi\circ h_i^{-1}}^2+\]
\[+M_{\phi}\cdot [(1-\Delta_{\hat{g}_i})^{-1},M_{\psi\circ h_i^{-1}}^2]\cdot M_{\psi\circ h_i^{-1}}^2.\]
Let $\theta\in C^{\infty}_c(\mathbb{R}^d)$ be such that $\theta\cdot (\psi\circ h_i^{-1})=\psi\circ h_i^{-1}.$ We have
\[[(1-\Delta_{\hat{g}_i})^{-1},M_{\psi\circ h_i^{-1}}^2]=(1-\Delta_{\hat{g}_i})^{-1}[\Delta_{\hat{g}_i},M_{\psi\circ h_i^{-1}}^2](1-\Delta_{\hat{g}_i})^{-1}=\]
\[=(1-\Delta_{\hat{g}_i})^{-1}M_{\theta}\cdot [\Delta_{\hat{g}_i},M_{\psi\circ h_i^{-1}}^2](1-\Delta_{\hat{g}_i})^{-1}\stackrel{L.\ref{do power lemma}}{\in}\mathcal{L}_{\frac{d}{2},\infty}\cdot \mathcal{L}_{d,\infty}=\mathcal{L}_{\frac{d}{3},\infty}.\]
Combining the last three formulae, we complete the proof.
\end{proof}

\begin{lem}\label{ctt power lemma} Let $(X,G)$ be a compact Riemannian manifold. If $0\leq\psi\in C^{\infty}(X),$ then
\[M_{\psi}^d(1-\Delta_G)^{-\frac{d}{2}}M_{\psi}^d-\Big(M_{\psi}^2(1-\Delta_G)^{-1}M_{\psi}^2\Big)^{\frac{d}{2}}\in\mathcal{L}_{\frac{2d}{2d+1},\infty}.\]
The same assertion holds for $\Delta_g,$ where $g$ is as in Theorem \ref{lb rd self-adjoint} and for $\psi\in C_c^{\infty}(\mathbb{R}^d).$
\end{lem}
\begin{proof} \textbf{Step 1:} We prove by induction that
\begin{equation}\label{ctt power induction eq}
M_{\psi}^{2n}(1-\Delta_G)^{-n}M_{\psi}^{2n}-(M_{\psi}^2(1-\Delta_G)^{-1}M_{\psi}^2)^n\in\mathcal{L}_{\frac{d}{2n+1},\infty},\quad n\geq0.
\end{equation}

Base of induction (i.e., the case $n=1$) is obvious. It remains to prove step of induction. Suppose \eqref{ctt power induction eq} holds for $n$ and let us prove it for $n+1.$ We write
\[(M_{\psi}^2(1-\Delta_G)^{-1}M_{\psi}^2)^{n+1}-M_{\psi}^{2n+2}(1-\Delta_G)^{-n-1}M_{\psi}^{2n+2}=\]
\[=M_{\psi}^2(1-\Delta_G)^{-1}M_{\psi}^2\cdot \Big((M_{\psi}^2(1-\Delta_G)^{-1}M_{\psi}^2)^n-M_{\psi}^{2n}(1-\Delta_G)^{-n}M_{\psi}^{2n}\Big)+\]
\[+M_{\psi}^2(1-\Delta_G)^{-\frac32}\cdot (1-\Delta_G)^{\frac12}[\Delta_G,M_{\psi}^{2n}](1-\Delta_G)^{-1}\cdot M_{\psi}^2(1-\Delta_G)^{-n}M_{\psi}^{2n}-\]
\[-M_{\psi}^{2n+2}(1-\Delta_G)^{-n-1}\cdot [M_{\psi}^2,(1-\Delta_G)^n](1-\Delta_G)^{\frac12-n}\cdot (1-\Delta_G)^{-\frac12}M_{\psi}^{2n}.\]
The first term on the right hand side belongs to $\mathcal{L}_{\frac{d}{2n+3},\infty}$ by inductive assumption and H\"older inequality. Note that the operators
\[(1-\Delta_G)^{\frac12}[\Delta_G,M_{\psi}^{2n}](1-\Delta_G)^{-1},\quad [M_{\psi}^2,(1-\Delta_G)^n](1-\Delta_G)^{\frac12-n},\]
are bounded. Hence, the second and third terms on the right hand side belong to $\mathcal{L}_{\frac{d}{2n+3},\infty}$ by  H\"older inequality. This establishes the step of induction and, hence, proves the claim in Step 1.

\textbf{Step 2:} Note that
\[M_{\psi}^d(1-\Delta_G)^{-\frac{d}{2}}M_{\psi}^d-M_{\psi}^{2d}(1-\Delta_G)^{-\frac{d}{2}}=\]
\[=M_{\psi}^d\cdot [M_{\psi}^d,(1-\Delta_G)^{-\frac{d}{2}}](1-\Delta_G)^{\frac{d+1}{2}}\cdot (1-\Delta_G)^{-\frac{d+1}{2}}.\]
Since the operator
\[[M_{\psi}^d,(1-\Delta_G)^{-\frac{d}{2}}](1-\Delta_G)^{\frac{d+1}{2}}\]
is bounded, it follows that
\[M_{\psi}^d(1-\Delta_G)^{-\frac{d}{2}}M_{\psi}^d-M_{\psi}^{2d}(1-\Delta_G)^{-\frac{d}{2}}\in\mathcal{L}_{\frac{d}{d+1},\infty}.\]
Taking adjoints, we obtain
\[M_{\psi}^d(1-\Delta_G)^{-\frac{d}{2}}M_{\psi}^d-(1-\Delta_G)^{-\frac{d}{2}}M_{\psi}^{2d}\in\mathcal{L}_{\frac{d}{d+1},\infty}.\]
Therefore,
\begin{equation}\label{cpl eq0}
\Big(M_{\psi}^d(1-\Delta_G)^{-\frac{d}{2}}M_{\psi}^d\Big)^2-M_{\psi}^{2d}(1-\Delta_G)^{-d}M_{\psi}^{2d}\in\mathcal{L}_{\frac{d}{2d+1},\infty}.
\end{equation}
Applying \eqref{ctt power induction eq} with $n=d$ and using \eqref{cpl eq0}, we obtain 
\[\Big(M_{\psi}^d(1-\Delta_G)^{-\frac{d}{2}}M_{\psi}^d\Big)^2-\Big(M_{\psi}^2(1-\Delta_G)^{-1}M_{\psi}^2\Big)^d\in\mathcal{L}_{\frac{d}{2d+1},\infty}.\]
The assertion follows now from Birman-Koplienko-Solomyak inequality.
\end{proof}

We remind the reader the following version of Connes Trace Theorem on Euclidean space established in \cite{DAO1}.

\begin{thm}\label{ctt rd} Let $\varphi$ be a normalised continuous trace on $\mathcal{L}_{1,\infty}.$ If $T\in\Pi$ is compactly supported from the right (i.e., there exists $\phi\in C^{\infty}_c(\mathbb{R}^d)$ such that $T=T\pi_1(\phi)$), then 
\[\varphi(T(1-\Delta)^{-\frac{d}{2}})=c_d'\int_{\mathbb{R}^d\times\mathbb{S}^{d-1}}\mathrm{sym}(T)dm,\]
where $m$ is the product of Lebesgue measure on $\mathbb{R}^d$ and Haar measure on $\mathbb{S}^{d-1}.$
\end{thm}

\begin{lem}\label{ctt main lemma} Let $(X,G)$ be a compact Riemannian manifold. Let $T\in\Pi_X$ be compactly supported in the chart $(\mathcal{U}_i,h_i).$ Let $\varphi$ be a continuous normalised trace on $\mathcal{L}_{1,\infty}.$ We have 
\[\varphi(T(1-\Delta_G)^{-\frac{d}{2}})=c_d\int_{\Omega_i\times\mathbb{R}^d}\mathrm{sym}(T_i)(t,\frac{s}{|s|})e^{-q_i(t,s)}dtds,\quad T_i=\mathrm{Ext}_{\Omega_i}(W_iTW_i^{-1}).\]
Here, $q_i$ is as in Notation \ref{qi nota}.
\end{lem}
\begin{proof} Fix $0\leq\psi\in C^{\infty}(X)$ such that $T=M_{\psi}TM_{\psi}$ and such that $\psi$ is compactly supported in $\mathcal{U}_i.$ 
	
By the tracial property we have
\[\varphi(T(1-\Delta_G)^{-\frac{d}{2}})=\varphi(M_{\psi}^dTM_{\psi}^d(1-\Delta_G)^{-\frac{d}{2}})=\varphi(TM_{\psi}^d(1-\Delta_G)^{-\frac{d}{2}}M_{\psi}^d).\]
Since $\varphi$ vanishes on $\mathcal{L}_{\frac{2d}{2d+1},\infty},$ it follows from Lemma \ref{ctt power lemma} that
\[\varphi(T(1-\Delta_G)^{-\frac{d}{2}})=\varphi(TA^{\frac{d}{2}}),\quad A=M_{\psi}^2(1-\Delta_G)^{-1}M_{\psi}^2.\]
Since both operators $T$ and $A$ are compactly supported in the chart $(\mathcal{U}_i,h_i),$ it follows that
\[\varphi(TA^{\frac{d}{2}})=\varphi(W_iTW_i^{-1}\cdot (W_iAW_i^{-1})^{\frac{d}{2}})=\]
\[=\varphi(\mathrm{Ext}_{\Omega_i}(W_iTW_i^{-1})\cdot(\mathrm{Ext}_{\Omega_i}( W_iAW_i^{-1}))^{\frac{d}{2}}).\]
	
Denote for brevity
\[B_i=M_{\psi\circ h_i^{-1}}^2(1-\Delta_{\hat{g}_i})^{-1}M_{\psi\circ h_i^{-1}}^2.\]
By Lemma \ref{ctt simplification lemma} and Birman-Koplienko-Solomyak inequality, we have
\[(\mathrm{Ext}_{\Omega_i}( W_iAW_i^{-1}))^{\frac{d}{2}}-B_i^{\frac{d}{2}}\in\mathcal{L}_{\frac{2d}{2d+1},\infty}.\]
Since $\varphi$ vanishes on $\mathcal{L}_{\frac{2d}{2d+1},\infty},$ it follows
\[\varphi(TA^{\frac{d}{2}})=\varphi(T_iB_i^{\frac{d}{2}}).\]
Using the second assertion in Lemma \ref{ctt power lemma}, we obtain
\[\varphi(TA^{\frac{d}{2}})=\varphi(T_i M_{\psi\circ h_i^{-1}}^d(1-\Delta_{\hat{g}_i})^{-\frac{d}{2}}M_{\psi\circ h_i^{-1}}^d)=\]
\[=\varphi(T_i(1-\Delta)^{-\frac{d}{2}}X_i)=\varphi(X_iT_i(1-\Delta)^{-\frac{d}{2}}),\]
where
\[X_i=(1-\Delta)^{\frac{d}{2}}(1-\Delta_{\hat{g}_i})^{-\frac{d}{2}}M_{\psi\circ h_i^{-1}}^d.\]

By Lemma \ref{do power psymbol lemma}, we have $X_i\in\Pi.$ Hence, the operator $X_iT_i\in\Pi$ is compactly supported from the right. By Theorem \ref{ctt rd} we have
\[\varphi(T(1-\Delta_g)^{-\frac{d}{2}})=\varphi(TA^{\frac{d}{2}})=c_d'\int_{\mathbb{R}^d\times\mathbb{S}^{d-1}}\mathrm{sym}(X_iT_i)dm,\]
where $m$ is the product of Lebesgue measure on $\mathbb{R}^d$ and Haar measure on $\mathbb{S}^{d-1}.$ By Lemma \ref{do power psymbol lemma}, we have
\[\mathrm{sym}(X_iT_i)(t,s)=\mathrm{sym}(T_i)(t,s)\cdot \psi(h_i^{-1}(t))^d\cdot\langle \hat{g}_i(t)^{-1}s,s\rangle^{-\frac{d}{2}}=\]
\[=\mathrm{sym}(T_i)(t,s)\cdot \langle g_i(t)^{-1}s,s\rangle^{-\frac{d}{2}}.\]
Thus,
\[\varphi(T(1-\Delta_g)^{-\frac{d}{2}})=c_d'\int_{\mathbb{R}^d\times\mathbb{S}^{d-1}}\mathrm{sym}(T_i)(t,s)\cdot\langle g_i(t)^{-1}s,s\rangle^{-\frac{d}{2}}dtds.\]
By passing to spherical coordinates we obtain
\[\int_{\mathbb{R}^d\times\mathbb{S}^{d-1}}\mathrm{sym}(T_i)(t,s)\cdot\langle g_i(t)^{-1}s,s\rangle^{-\frac{d}{2}}dtds=\]
\[=c_d''\int_{\mathbb{R}^d\times\mathbb{R}^d}\mathrm{sym}(T_i)(t,\frac{s}{|s|})e^{-q_i(t,s)}dtds.\]
Combining these $2$ equalities, we complete the proof.
\end{proof}

\begin{proof}[Proof of Theorem \ref{ctt manifold}] Suppose first that $T\in\Pi_X$ is compactly supported in the chart $(\mathcal{U}_i,h_i).$ By Lemma \ref{ctt main lemma}, we have
\[\varphi(T(1-\Delta_G)^{-\frac{d}{2}})=c_d\int_{\Omega_i\times\mathbb{R}^d}\mathrm{sym}(T_i)(t,\frac{s}{|s|})\cdot e^{-q_i(t,s)}dtds.\]
By \eqref{liouville measure defining eq} and \eqref{qi comptibility riemann lemma}, we have
\[\int_{\Omega_i\times\mathbb{R}^d}\mathrm{sym}(T_i)(t,\frac{s}{|s|})\cdot e^{-q_i(t,s)}dtds=\int_{T^{\ast}X}\mathrm{sym}_X(T)e^{-q_X}d\lambda.\]
A combination of these two equalities proves the assertion for $T$ compacly supported in some $\mathcal{U}_i.$

Let now $T\in\Pi_X$ be arbitrary. Let $(\phi_n)_{n=1}^N$ be a fixed good partition of unity. 
	
We write
\[T=\sum_{n=0}^NT_n,\quad T_0=\sum_{m=1}^NM_{\phi_m}^{\frac12}\cdot [M_{\phi_m}^{\frac12},T],\quad T_n=M_{\phi_n}^{\frac12}TM_{\phi_n}^{\frac12},\quad n\geq1.\]	
By assumption, $[T,M_{\psi}]$ is compact for every $\psi\in C(X).$ In particular, $T_0$ is compact. Thus,
\[\varphi(T_0(1-\Delta_G)^{-\frac{d}{2}})=0,\quad \mathrm{sym}_X(T_0)=0.\]
By the first paragraph, we have
\[\varphi(T(1-\Delta_G)^{-\frac{d}{2}})=\sum_{n=0}^N\varphi(T_n(1-\Delta_G)^{-\frac{d}{2}})=\]
\[=\sum_{n=1}^Nc_d\int_{T^{\ast}X}\mathrm{sym}_X(T_n)e^{-q_X}d\lambda=c_d\int_{T^{\ast}X}\mathrm{sym}_X(T)e^{-q_X}d\lambda.\]
\end{proof}

\appendix

\section{Proof of globalisation theorem}\label{appendix section}

We prove Theorem \ref{globalisation theorem} in the following series of lemmas.

\begin{lem}\label{A is a starsubalgebra} In the setting of Definitions \ref{local algebras def} and \ref{global algebra def},  $\mathcal{A}$ is a unital $\ast$-subalgebra in $B(L_2(X,\nu)).$
\end{lem}
\begin{proof} Suppose $T,S\in\mathcal{A}.$ It is immediate that $1,T+S,T^{\ast}\in\mathcal{A}.$ It suffices to show that also $TS\in\mathcal{A}.$
	
Let $i\in\mathbb{I}$ and let $0\leq \phi\in C_c(\mathcal{U}_i).$ We write
\[M_{\phi}TSM_{\phi}=M_{\phi^{\frac12}}TM_{\phi^{\frac12}}\cdot M_{\phi^{\frac12}}SM_{\phi^{\frac12}}+M_{\phi^{\frac12}}\cdot [M_{\phi}^{\frac12},T]\cdot SM_{\phi}+M_{\phi}^{\frac12}TM_{\phi}^{\frac12}\cdot[S,M_{\phi}^{\frac12}]M_{\phi}^{\frac12}.\]
By Definition \ref{global algebra def} \eqref{gada}, we have
\[M_{\phi^{\frac12}}TM_{\phi^{\frac12}},M_{\phi^{\frac12}}SM_{\phi^{\frac12}}\in\mathcal{A}_i.\]
Since $\mathcal{A}_i$ is a subalgebra, it follows that
\[M_{\phi^{\frac12}}TM_{\phi^{\frac12}}\cdot M_{\phi^{\frac12}}SM_{\phi^{\frac12}}\in\mathcal{A}_i.\]
By Definition \ref{global algebra def} \eqref{gadb}, the operators $[M_{\phi}^{\frac12},T]$ and $[S,M_{\phi}^{\frac12}]$ are compact. Therefore,
\[M_{\phi^{\frac12}}\cdot [M_{\phi}^{\frac12},T]\cdot SM_{\phi}+M_{\phi}^{\frac12}TM_{\phi}^{\frac12}\cdot[S,M_{\phi}^{\frac12}]M_{\phi}^{\frac12}\]
is compact. However, the latter operator is compactly supported in $\mathcal{U}_i.$ By Definition \ref{local algebras def} \eqref{ladd}, we have
\[M_{\phi^{\frac12}}\cdot [M_{\phi}^{\frac12},T]\cdot SM_{\phi}+M_{\phi}^{\frac12}TM_{\phi}^{\frac12}\cdot[S,M_{\phi}^{\frac12}]M_{\phi}^{\frac12}\in\mathcal{A}_i.\]
Therefore,
\[M_{\phi}TSM_{\phi}\in\mathcal{A}_i,\quad \phi\in C_c(\mathcal{U}_i),\quad i\in\mathbb{I}.\]
Since also
\[[TS,M_{\psi}]=T\cdot [S,M_{\psi}]+[T,M_{\psi}]\cdot S\]
is compact for every $\psi\in C(X),$ it follows that $TS\in\mathcal{A}.$
\end{proof}

\begin{lem}\label{A is a cstarsubalgebra} In the setting of Definitions \ref{local algebras def} and \ref{global algebra def},  $\mathcal{A}$ is a unital $C^{\ast}$-subalgebra in $B(L_2(X,\nu)).$
\end{lem}
\begin{proof} It is established in Lemma \ref{A is a starsubalgebra} that $\mathcal{A}$ is a unital $\ast$-subalgebra in $B(L_2(X,\nu)).$ It suffices to show that $\mathcal{A}$ is closed in the uniform norm.
	
Let $\{T_n\}_{n\geq1}\subset\mathcal{A}$ and let $T\in B(L_2(X,\nu))$ be such that $T_n\to T$ in the uniform norm. Let us show that $T\in\mathcal{A}.$
	
Let $i\in\mathbb{I}$ and let $\phi\in C_c(\mathcal{U}_i).$ Take $\phi_0\in C_c(\mathcal{U}_i)$ such that $\phi\phi_0=\phi.$ We have $M_{\phi}T_nM_{\phi}\in\mathcal{A}_i$ and, therefore,
\[M_{\phi}T_nM_{\phi}=M_{\phi_0}\cdot M_{\phi}T_nM_{\phi}\cdot M_{\phi_0}\in M_{\phi_0}\mathcal{A}_iM_{\phi_0},\quad n\geq 1.\]
By assumption, $M_{\phi}T_nM_{\phi}\to M_{\phi}TM_{\phi}$ in the uniform norm. Hence, $M_{\phi}TM_{\phi}$ belongs to the closure of $M_{\phi_0}\mathcal{A}_iM_{\phi_0}$ in the uniform norm. By Definition \ref{local algebras def} \eqref{ladf}, $M_{\phi}TM_{\phi}\in\mathcal{A}_i.$
	
If $\psi\in C(X),$ then
\[[T,M_{\psi}]=\lim_{n\to\infty}[T_n,M_{\psi}],\]
is the limit of compact operators in the uniform norm and is, therefore, compact.
	
Combining the results in the preceding paragraphs, we conclude that $T\in\mathcal{A}.$ This completes the proof. 
\end{proof}

\begin{proof}[Proof of Theorem \ref{globalisation theorem} \eqref{gta}] We already demonstrated in Lemma \ref{A is a cstarsubalgebra} that $\mathcal{A}$ is a unital $C^{\ast}$-subalgebra in $B(L_2(X,\nu)).$ It remains to show that $\mathcal{A}_i\subset\mathcal{A}$ for every $i\in\mathbb{I}$ and that $\mathcal{K}(L_2(X,\nu))\subset\mathcal{A}.$

Let $T\in\mathcal{A}_i$ and $\phi\in C_c(\mathcal{U}_j)$. By Definition \ref{local algebras def} \eqref{ladb}, $T$ is compactly supported in $\mathcal{U}_i.$ Choose $\phi_0\in C_c(\mathcal{U}_i)$ such that $T=M_{\phi_0}T=TM_{\phi_0}.$ Then $\phi\phi_0\in C_c(\mathcal{U}_i\cap \mathcal{U}_j)$ and $M_{\phi\phi_0}\in\mathcal{A}_i.$ Since $\mathcal{A}_i$ is an algebra, it follows that  $M_{\phi}TM_{\phi}=M_{\phi\phi_0}TM_{\phi\phi_0}\in\mathcal{A}_i.$ Since $M_{\phi}TM_{\phi}$ is compactly supported in $\mathcal{U}_i\cap\mathcal{U}_j,$ it follows from Definition \ref{local algebras def} \eqref{ladc} that $M_{\phi}TM_{\phi}\in\mathcal{A}_j.$  This verifies the condition \eqref{gada} in Definition \ref{global algebra def} for the operator $T.$ Let $\psi\in C(X).$ As above, choose $\phi_0\in C_c(\mathcal{U}_i)$ such that $T=M_{\phi_0}T=TM_{\phi_0}.$ It follows that $[T,M_{\psi}]=[T,M_{\psi\phi_0}].$ Since $\psi\phi_0\in C_c(\mathcal{U}_i),$ it follows from the condition \eqref{ladg} in Definition \ref{local algebras def} that $[T,M_{\psi}]$ is compact.  This verifies the condition \eqref{gadb} in Definition \ref{global algebra def} for the operator $T.$ Hence, $T\in\mathcal{A}$ and, therefore, $\mathcal{A}_i\subset\mathcal{A}.$
	
Let $T\in \mathcal{K}(L_2(X,\nu)).$ For every $i\in\mathbb{I}$ and for every $\phi\in C_c(\mathcal{U}_i),$ the operator $M_{\phi}TM_{\phi}$ is simultaneously compact and compactly supported in $\mathcal{U}_i.$ By Definition \ref{local algebras def} \eqref{ladd}, we have $M_{\phi}TM_{\phi}\in\mathcal{A}_i.$ Clearly, $[T,M_{\psi}]\in \mathcal{K}(L_2(X,\nu))$ for every $\psi\in C(X).$ Therefore, $T\in\mathcal{A}.$ Since $T\in \mathcal{K}(L_2(X,\nu))$ is arbitrary, it follows that $\mathcal{K}(L_2(X,\nu))\subset\mathcal{A}.$
\end{proof}

The purpose of the next lemma is twofold: to establish Theorem \ref{globalisation theorem} \eqref{gtc} and to provide a concrete form of $\mathrm{hom}$ used in the proof of Theorem \ref{globalisation theorem} \eqref{gtb}. 

\begin{lem}\label{global homomorphism def} Suppose we are in the setting of Theorem \ref{globalisation theorem}. Let $(\phi_n)_{n=1}^N$ be a good\footnote{See Definition \ref{good partition of unity}.} partition of unity so that $\phi_n$ is compactly supported in $\mathcal{U}_{i_n}$ for $1\leq n\leq N.$ If $\mathrm{hom}:\mathcal{A}\to\mathcal{B}$ is a $\ast$-homomorphism as in Theorem \ref{globalisation theorem} \eqref{gtb}, then
\begin{equation}\label{hom concrete expr}
\mathrm{hom}(T)=\sum_{n=1}^N\mathrm{hom}_{i_n}(M_{\phi_n^{\frac12}}TM_{\phi_n^{\frac12}}),\quad T\in\mathcal{A}.
\end{equation}
In particular, $\mathrm{hom}$ is unique.
\end{lem}
\begin{proof} For every $T\in B(L_2(X,\nu)),$ we write
\[T=\sum_{n=0}^NT_n,\quad T_n=M_{\phi_n^{\frac12}}TM_{\phi_n^{\frac12}},\quad 1\leq n\leq N,\quad T_0=\sum_{k=1}^NM_{\phi_k^{\frac12}}\cdot [M_{\phi_k^{\frac12}},T].\]
Every $T_n,$ $1\leq n\leq N,$ is compactly supported in the chart $(\mathcal{U}_{i_n},h_{i_n}).$ If $T\in\mathcal{A},$ then $T_n\in\mathcal{A}_{i_n}$ for $1\leq n\leq N.$ Hence,
\[\mathrm{hom}(T_n)=\mathrm{hom}_{i_n}(T_n),\quad 1\leq n\leq N.\]
If $T\in\mathcal{A},$ then $T_0$ is compact by Definition \ref{global algebra def} \eqref{gadb}. Since $\mathrm{hom}$ vanish on compact operators, it follows that
\[\mathrm{hom}(T_0)=0.\]
Thus,
\[\mathrm{hom}(T)=\sum_{n=0}^N\mathrm{hom}(T_n)=\sum_{n=1}^N\mathrm{hom}_{i_n}(M_{\phi_n^{\frac12}}TM_{\phi_n^{\frac12}}),\quad T\in\mathcal{A}.\]
\end{proof}

We now fix a good partition of unity and prove that the concrete map $\mathrm{hom}$ introduced in Lemma \ref{global homomorphism def} is $\ast$-homomorphism.

\begin{lem}\label{zeroth homomorphism lemma} Let $\mathrm{hom}$ be the mapping on the right hand side in \eqref{hom concrete expr}. We have 
\[\mathrm{hom}(TM_{\phi})=\mathrm{hom}(T)\cdot \mathrm{hom}(M_{\phi}),\quad T\in\mathcal{A},\quad \phi\in C(X).\]	
\end{lem}
\begin{proof} Let $\psi_n\in C_c(\mathcal{U}_{i_n})$ be such that $\phi_n\psi_n=\phi_n.$ We write
\[M_{\phi_n^{\frac12}}TM_{\phi}M_{\phi_n^{\frac12}}=M_{\phi_n^{\frac12}}TM_{\phi_n^{\frac12}}\cdot M_{\phi\psi_n}.\]
Since $\mathrm{hom}_{i_n}$ is a homomorphism, it follows that
\[\mathrm{hom}_{i_n}(M_{\phi_n^{\frac12}}TM_{\phi}M_{\phi_n^{\frac12}})=\mathrm{hom}_{i_n}(M_{\phi_n^{\frac12}}TM_{\phi_n^{\frac12}})\cdot \mathrm{hom}_{i_n}(M_{\phi\psi_n}).\]
It follows from Definition \ref{local homomorphism def} \eqref{lhdd} and \eqref{hom concrete expr} that
\[\mathrm{hom}_{i_n}(M_{\phi\psi_n})=\mathrm{Hom}(\phi\psi_n)=\mathrm{Hom}(\phi)\cdot\mathrm{Hom}(\psi_n)=\mathrm{hom}_{i_n}(M_{\psi_n})\cdot \mathrm{hom}(M_{\phi}).\]
Again using the fact that $\mathrm{hom}_{i_n}$ is a homomorphism, we obtain
\[\mathrm{hom}_{i_n}(M_{\phi_n^{\frac12}}TM_{\phi}M_{\phi_n^{\frac12}})=\mathrm{hom}_{i_n}(M_{\phi_n^{\frac12}}TM_{\phi_n^{\frac12}}\cdot M_{\psi_n})\cdot \mathrm{hom}(M_{\phi}).\]
However, by the choice of $\psi_n,$ we have
\[M_{\phi_n^{\frac12}}TM_{\phi_n^{\frac12}}\cdot M_{\psi_n}=M_{\phi_n^{\frac12}}TM_{\phi_n^{\frac12}}.\]
Therefore,
\[\mathrm{hom}_{i_n}(M_{\phi_n^{\frac12}}TM_{\phi}M_{\phi_n^{\frac12}})=\mathrm{hom}_{i_n}(M_{\phi_n^{\frac12}}TM_{\phi_n^{\frac12}})\cdot \mathrm{hom}(M_{\phi}),\quad 1\leq n\leq N.\]
Summing over $1\leq n\leq N,$ we complete the proof.
\end{proof}

\begin{lem}\label{hom vs homj lemma} Let $\mathrm{hom}$ be the mapping on the right hand side in \eqref{hom concrete expr}. We have $\mathrm{hom}=\mathrm{hom}_j$ on $\mathcal{A}_j$ for every $j\in\mathbb{I}.$
\end{lem}
\begin{proof} Let $T\in\mathcal{A}_j$ and let $(\phi_n)_{n=1}^N$ be as in Lemma \ref{global homomorphism def}. For every $1\leq n\leq N,$ the operator $M_{\phi_n^{\frac12}}TM_{\phi_n^{\frac12}}$ is compactly supported both in the chart $(\mathcal{U}_j,h_j)$ and in the chart $(\mathcal{U}_{i_n},h_{i_n}).$ By Definition \ref{local homomorphism def} \eqref{lhdb}, we have
\[\mathrm{hom}_{i_n}(M_{\phi_n^{\frac12}}TM_{\phi_n^{\frac12}})=\mathrm{hom}_j(M_{\phi_n^{\frac12}}TM_{\phi_n^{\frac12}}).\]
Let $\phi\in C_c(\mathcal{U}_j)$ be such that $T=M_{\phi}T=TM_{\phi}.$ Since $\mathrm{hom}_j$ is a homomorphism, it follows that
\[\mathrm{hom}_j(M_{\phi_n^{\frac12}}TM_{\phi_n^{\frac12}})=\mathrm{hom}_j(M_{\phi_n^{\frac12}\phi}\cdot T\cdot M_{\phi\phi_n^{\frac12}})=\]
\[=\mathrm{hom}_j(T)\cdot \mathrm{hom}_j(M_{\phi_n^{\frac12}\phi})\cdot \mathrm{hom}_j(M_{\phi\phi_n^{\frac12}})=\mathrm{hom}_j(T)\cdot \mathrm{hom}_j(M_{\phi_n\phi^2}).\]
Therefore,
\[\mathrm{hom}_{i_n}(M_{\phi_n^{\frac12}}TM_{\phi_n^{\frac12}})=\mathrm{hom}_j(T)\cdot \mathrm{hom}_j(M_{\phi_n\phi^2}),\quad 1\leq n\leq N.\]
Summing over $1\leq n\leq N,$ we obtain
\[\mathrm{hom}(T)=\mathrm{hom}_j(T)\cdot \big(\sum_{n=1}^N\mathrm{hom}_j(M_{\phi_n\phi^2})\big)=\mathrm{hom}_j(T)\cdot \mathrm{hom}_j(M_{\phi^2}).\]
Again using the fact the $\mathrm{hom}_j$ is a homomorphism, we obtain 
\[\mathrm{hom}(T)=\mathrm{hom}_j(TM_{\phi^2}).\]
Taking into account that $TM_{\phi^2}=T,$ we complete the proof.
\end{proof}

\begin{lem}\label{first homomorphism lemma} We have
\[\mathrm{hom}(TS)=\mathrm{hom}(T)\cdot \mathrm{hom}(S),\quad T,S\in\mathcal{A}_j,\quad j\in\mathbb{I}.\]
\end{lem}
\begin{proof} Using Lemma \ref{hom vs homj lemma} and taking into account that $\mathrm{hom}_j$ is a homomorphism, we write
\[\mathrm{hom}(TS)=\mathrm{hom}_j(TS)=\mathrm{hom}_j(T)\cdot\mathrm{hom}_j(S)=\mathrm{hom}(T)\cdot\mathrm{hom}(S).\]
\end{proof}

\begin{lem}\label{second homomorphism lemma} If $T,S\in\mathcal{A}$ and if $T$ is compactly supported in the chart $(\mathcal{U}_j,h_j),$ then
\[\mathrm{hom}(TS)=\mathrm{hom}(T)\cdot \mathrm{hom}(S).\]
\end{lem}
\begin{proof} Let $\phi\in C_c(\mathcal{U}_j)$ be such that $T=TM_{\phi}.$ We write
\[TS=TS_1+TS_2,\quad S=M_{\phi^{\frac12}}SM_{\phi^{\frac12}},\quad S_2=M_{\phi^{\frac12}}[M_{\phi^{\frac12}},S].\]
By Definition \ref{global algebra def} \eqref{gadb}, $S_2$ is compact.  By construction, $\mathrm{hom}$ vanishes on compact operators. It follows that
\[\mathrm{hom}(TS)=\mathrm{hom}(TS_1).\]
Since $T$ and $S_1$ are compactly supported in the chart $(\mathcal{U}_j,h_j),$ it follows from Lemma \ref{first homomorphism lemma} that
\[\mathrm{hom}(TS)=\mathrm{hom}(T)\cdot \mathrm{hom}(S_1).\]
By Lemma \ref{zeroth homomorphism lemma}, we have
\[\mathrm{hom}(TS)=\mathrm{hom}(T)\cdot \mathrm{hom}(S)\cdot \mathrm{hom}(M_{\phi^{\frac12}})^2=\]
\[=\mathrm{hom}(T)\cdot \mathrm{hom}(M_{\phi})\cdot \mathrm{hom}(S)=\mathrm{hom}(TM_{\phi})\cdot \mathrm{hom}(S).\]
Since $T=TM_{\phi},$ the assertion follows.
\end{proof}

\begin{proof}[Proof of Theorem \ref{globalisation theorem} \eqref{gtb}] It is immediate that $\mathrm{hom}$ is a linear $\ast$-preserving mapping. We now prove that $\mathrm{hom}$ preserves multiplication.
	
Let $T,S\in\mathcal{A}$ and let $(\phi_n)_{n=1}^N$ be as in Lemma \ref{global homomorphism def}. We write $T=\sum_{n=0}^NT_n,$ where	\[T_n=M_{\phi_n^{\frac12}}TM_{\phi_n^{\frac12}},\quad 1\leq n\leq N,\quad T_0=\sum_{k=1}^NM_{\phi_k^{\frac12}}[M_{\phi_k^{\frac12}},T].\]
Every $T_n,$ $1\leq n\leq N,$ is compactly supported in the chart $(\mathcal{U}_{i_n},h_{i_n}).$ By Lemma \ref{second homomorphism lemma}, we have
\[\mathrm{hom}(T_nS)=\mathrm{hom}(T_n)\cdot\mathrm{hom}(S),\quad 1\leq n\leq N.\]
The operators $T_0$ and $T_0S$ are compact. By construction, $\mathrm{hom}$ vanishes on compact operators. It follows that
\[\mathrm{hom}(T_0S)=0=0\cdot\mathrm{hom}(S)=\mathrm{hom}(T_0)\cdot\mathrm{hom}(S).\]
By linearity, we have
\[\mathrm{hom}(TS)=\sum_{n=0}^N\mathrm{hom}(T_nS)=\sum_{n=0}^N\mathrm{hom}(T_n)\cdot \mathrm{hom}(S)=\mathrm{hom}(T)\cdot \mathrm{hom}(S).\]
Thus, $\mathrm{hom}$ is a $\ast$-homomorphism.

Let us now show that $\mathrm{ker}(\mathrm{hom})=\mathcal{K}(L_2(X,\nu)).$ If $T\in\mathcal{A}$ is such that $\mathrm{hom}(T)=0,$ then $\mathrm{hom}(T^{\ast}T)=0.$ By construction of $\mathrm{hom},$ we have
\[\sum_{n=1}^n\mathrm{hom}_{i_n}(M_{\phi_n^{\frac12}}T^{\ast}TM_{\phi_n^{\frac12}})=0.\]
Every summand on the left hand side is positive. Therefore,
\[\mathrm{hom}_{i_n}(M_{\phi_n^{\frac12}}T^{\ast}TM_{\phi_n^{\frac12}})=0,\quad 1\leq n\leq N.\]
By Definition \ref{local homomorphism def} \eqref{lhdc}, we have
\[M_{\phi_n^{\frac12}}T^{\ast}TM_{\phi_n^{\frac12}}\in\mathcal{K}(L_2(X,\nu)),\quad 1\leq n\leq N.\]
In other words,
\[TM_{\phi_n^{\frac12}}\in\mathcal{K}(L_2(X,\nu)),\quad 1\leq n\leq N.\]
Multiplying on the right by $M_{\phi_n^{\frac12}}$ and summing over $1\leq n\leq N,$ we conclude that $T\in\mathcal{K}(L_2(X,\nu)).$
\end{proof}

Acknowledgement: The authors sincerely thank Professor N. Higson for the Question \ref{higson question} which was the starting point of this article. 
	
The special thanks are due to Professor A. Connes for his insightful comments on the earlier version of this manuscript. In particular, the observation that diffeomorphisms of $\mathbb{R}^d$ considered in this text extend to diffeomorphisms of $P^d(\mathbb{R})$ is due to Professor Connes.

Funding: The second author was partially supported by the ARC.


\begin{thebibliography}{99}
\bibitem{Adams-book} Adams R. \textit{Sobolev spaces.} Pure and Applied Mathematics, Vol. \textbf{65}. Academic Press, New York-London, 1975.
\bibitem{AM-book} Alberti P., Matthes R. \textit{ Connes' trace formula and Dirac realization of Maxwell and Yang-Mills action.} Noncommutative geometry and the standard model of elementary particle physics (Hesselberg, 1999), 40--74, Lecture Notes in Phys., 596, Springer, Berlin, 2002.
\bibitem{Atiyah-Singer} Atiyah M., Singer I. \textit{ The index of elliptic operators.  I.} Ann. of Math. (2) \textbf{ 87} (1968), 484--530.
\bibitem{Baum-Douglas} Baum P., Douglas R. \textit{ Toeplitz operators and Poincare duality.} Toeplitz centennial (Tel Aviv, 1981), pp. 137--166, Operator Theory: Advances and Applications, 4, Birkh\"{a}user, Basel-Boston, Mass., 1982. 
\bibitem{Blackadar-book} Blackadar B. \textit{ Operator algebras. Theory of $C^{\ast}$-algebras and von Neumann algebras.} Encyclopaedia of Mathematical Sciences, \textbf{ 122}. Operator Algebras and Non-commutative Geometry, III. Springer-Verlag, Berlin, 2006.
\bibitem{Chavel} Chavel I. \textit{ Riemannian geometry. A modern introduction.} Second edition. Cambridge Studies in Advanced Mathematics, \textbf{ 98}. Cambridge University Press, Cambridge, 2006. 
\bibitem{Connes-action} Connes A. \textit{ The action functional in noncommutative geometry.} Comm. Math. Phys. \textbf{ 117} (1988), no. 4, 673--683.
\bibitem{Cordes87} Cordes H. \textit{ Spectral Theory of Linear Differential Operators and Comparison Algebras.} London Mathematical Society Lecture Note Series \textbf{ 76}, Cambridge University Press, Cambridge, 1987.
\bibitem{Dixmier} Dixmier J. \textit{ $C^{\ast}$-algebras.} North-Holland Mathematical Library, Vol. 15. North-Holland Publishing Co., Amsterdam-New York-Oxford, 1977.
\bibitem{Gohberg} Gohberg I. \textit{ On the theory of multidimensional singular integral equations.} Dokl. Akad. Nauk SSSR \textbf{ 133} (1960), 1279--1282 (Russian); translated as Soviet Math. Dokl. \textbf{ 1} (1960), 960--963.
\bibitem{GVF-book} Gracia-Bondia J, Varilly J., Figueroa H. \textit{ Elements of noncommutative geometry.} Birkh\"auser Advanced Texts: Basler Lehrb\"ucher, Birkh\"auser Boston, Inc., Boston, MA, 2001.
\bibitem{Hormander67} H\"{o}rmander L. \textit{ Pseudo-differential operators and hypoelliptic equations.} Singular Integrals (Proc. Sympos. Pure Math., Chicago, Ill., 1966), pp. 138--183. Amer. Math. Soc., Providence, R.I., 1967.
\bibitem{Joshi} Joshi M. \textit{ Lectures on Pseudo-Differential Operators.} arXiv (1999), arXiv:math/9906155.
\bibitem{Kohn-Nirenberg} Kohn J., Nirenberg L. \textit{ An algebra of pseudo-differential operators.} Comm. Pure Appl. Math. \textbf{ 18} (1965), 269--305.
\bibitem{LeSZ} Levitina G., Sukochev F., Zanin D. \textit{ Cwikel estimates revisited.} 
Proc. Lond. Math. Soc. (3) \textbf{ 120} (2020), no. 2, 265--304.
\bibitem{LMSZ} Lord S., McDonald E., Sukochev F., Zanin D. \textit{ Quantum differentiability of essentially bounded functions on Euclidean space.} J. Funct. Anal. \textbf{ 273} (2017), no. 7, 2353--2387.
\bibitem{LMSZ-book} Lord S., McDonald E., Sukochev F., Zanin D. \textit{ Singular traces. Vol.2--Applications.} to appear.
\bibitem{LSZ-obzor} Lord S., Sukochev F., Zanin D. \textit{ Advances in Dixmier traces and applications.} Advances in noncommutative geometry, 491--583, Springer, Cham, 2019.
\bibitem{LSZ-book} Lord S., Sukochev F., Zanin D. \textit{ Singular traces. Vol. 1--Theory.} De Gruyter Studies in Mathematics, 46/1. De Gruyter, Berlin, 2021.
\bibitem{DAO2} McDonald E., Sukochev F., Zanin D. \textit{ A $C^{\ast}$-algebraic approach to the principal symbol. II.} Math. Ann. \textbf{ 374} (2019), no. 1-2, 273--322.
\bibitem{melo05} Melo S. \textit{ Norm closure of classical pseudodifferential operators does not contain H\"{o}rmander's class.} Geometric analysis of PDE and several complex variables, 329--336, Contemp. Math., \textbf{ 368}, Amer. Math. Soc., Providence, RI, 2005. 
\bibitem{michor} Michor P. \textit{ Topics in differential geometry.} American Mathematical Society, Providence, RI, 2008.
\bibitem{PalaisTAMS} Palais R. \textit{ Natural operations on differential forms.} Trans. Amer. Math. Soc. \textbf{ 92} (1959), 125--141.
\bibitem{RT-book} Ruzhansky M., Turunen V. \textit{ Pseudo-differential operators and symmetries.} Background analysis and advanced topics. Pseudo-Differential Operators. Theory and Applications, \textbf{ 2}. Birkh\"auser Verlag, Basel, 2010.
\bibitem{Sakai-book} Sakai S. \textit{ $C^{\ast}$-algebras and $W^{\ast}$-algebras.} Reprint of the 1971 edition. Classics in Mathematics. Springer-Verlag, Berlin, 1998.
\bibitem{Seeley} Seeley R. \textit{ Integro-differential operators on vector bundles.} Trans. Amer. Math. Soc. \textbf{ 117} (1965), 167--204. 
\bibitem{Simon-book}  Simon B. \textit{ Trace ideals and their applications.} Second edition. Mathematical Surveys and Monographs, \textbf{ 120}. American Mathematical Society, Providence, RI, 2005.
\bibitem{Stein-book} Stein E. \textit{ Harmonic analysis: real-variable methods, orthogonality, and oscillatory integrals.} Princeton Mathematical Series, \textbf{ 43}. Monographs in Harmonic Analysis, III. Princeton University Press, Princeton, NJ, 1993.
\bibitem{Strichartz} Strichartz R. \textit{ Analysis of the Laplacian on the complete Riemannian manifold.} J. Functional Analysis \textbf{ 52} (1983), no. 1, 48--79.
\bibitem{DAO1} Sukochev F., Zanin D. \textit{ A $C^{\ast}$-algebraic approach to the principal symbol. I.} J. Operator Theory \textbf{ 80} (2018), no. 2, 481--522. 
\bibitem{Taylor-book} Taylor M. \textit{ Partial differential equations I. Basic theory. Second edition.} Applied Mathematical Sciences, \textbf{ 115}. Springer, New York, 2011.
\bibitem{Treves1} Treves F. \textit{ Introduction to pseudodifferential and Fourier integral operators. Vol. 1. Pseudodifferential operators.} University Series in Mathematics. Plenum Press, New York-London, 1980.
\end{thebibliography}
\end{document}